\documentclass[10pt,letterpaper]{amsart}
\usepackage[title]{appendix}
\usepackage{amsfonts}
\usepackage{eufrak}
\usepackage{amsmath}
\usepackage{amsrefs}
\usepackage{amssymb}
\usepackage{amsthm}
\usepackage{array}
\usepackage{bm}
\usepackage{dcpic,pictexwd}
\usepackage{fancyhdr}
\usepackage{indentfirst}
\usepackage{ifthen}
\usepackage{accents}
\usepackage{scalerel}
\numberwithin{equation}{section}
\newtheorem{Theorem}{Theorem}[section]

\newtheorem{Lemma}[Theorem]{Lemma}
\newtheorem{Corollary}[Theorem]{Corollary}
\newtheorem{Definition}[Theorem]{Definition}
\theoremstyle{definition}
\newtheorem{Example}[Theorem]{Example}
\newtheorem{Remark}[Theorem]{Remark}

\newcommand\hook{\mathbin{\raise0.5pt\hbox{\hbox{{\vbox{\hrule height.4pt width6pt depth0pt}}}\vrule height6pt width.4pt depth0pt}\,}}
\newcommand\CR{{{\mathcal R}^\infty}}
\newcommand\reals{\mathbb{R}}
\newcommand\CL{{\bf L}}
\newcommand\LD{{\mathcal L}}

\newcommand\CD{{\mathcal D}}
\newcommand\CE{{\mathcal E}}
\newcommand\CF{{\mathcal F}}
\newcommand\CS{{\mathcal S}}
\newcommand\CZ{{\mathcal Z}}
\newcommand\pr{{\rm pr}}
\newcommand\CV{{\mathcal V}}
\newcommand\AQ{{Q\,}}
\newcommand\CI{{\mathcal I}}
\newcommand\CJ{{J^\infty(\reals^2,\reals)}}
\newcommand\EJ{J^\infty(\reals,\reals)}
\newcommand\sE{{\scaleto{E}{4pt}}}
\newcommand\sEL{{\scaleto{\mkern-2.1mu E}{4pt}}}
\newcommand\sH{{\scaleto{\mkern-1.5mu H}{4pt}}     }
\newcommand\sV{{\scaleto{\mkern-1.0mu V}{4pt}}}
\newcommand\tdV{{d _\sV}}
\newcommand\bFQ{{\bf F}_Q}

\newcommand\bLP{{\bf L}_P}

\newcommand\tQ{Q}
\newcommand\ts{s}
\newcommand\tr{r}
\newcommand\tR{R}
\newcommand\tL{L}
\newcommand\tP{P}
\newcommand\tH{H}
\newcommand\tg{g}

\newcommand\Eop{{\bf E}}

\newcommand\tdH{{d_\sH}}

\newcommand\bbeta{{\bm \beta}}

\title[The Cohomology of $u_t=K$]{Variational Operators, Symplectic Operators, and the Cohomology of Scalar Evolution Equations}                                     
\author{M.E. Fels}      
\author{E. Yasar}                 

\address{
Mark E. Fels\\               
Department of Mathematics and Statistics, Utah State University, Logan Utah, 84322\\            
mark.fels@usu.edu             
}

\address{
Emrullah Yasar\\               
Department of Mathematics, Uludag University \\            
eyasar@uludag.edu.tr        
}

\begin{document}

\maketitle

\begin{abstract}    For a scalar evolution equation $u_t=K(t,x,u,u_x,\ldots, u_n), n\geq 2$ the cohomology spaces $H^{1,s}(\CR)$  vanishes for $s\geq 3$ while the space $H^{1,2}(\CR)$ is isomorphic to the space of variational operators. The cohomology space $H^{1,2}(\CR)$ is also shown to be isomorphic to the space of symplectic operators for $u_t=K$ for which the equation is Hamiltonian. Third order scalar evolution equations admitting a first order symplectic (or variational) operator are characterized. The symplectic nature of the potential form of a bi-Hamiltonian evolution equation is also presented. 
\end{abstract}

\section{Introduction}

Given a scalar differential equation $\Delta =0$, the multiplier problem in the calculus of variations consists in determining whether there exists a smooth function $m$ (the multiplier) and a smooth Lagrangian $L$ such that
\begin{equation}
m \cdot \Delta  = \Eop (L)
\label{InvProb}
\end{equation}
where $\Eop$ is the Euler-Lagrange operator and $\Eop(L)$ is the Euler-Lagrange expression for $L$. The problem of determining whether $m$ and $L$ exists has a long history and is known as the inverse problem in the calculus of variations  \cite{anderson-thompson:1992a, fels:1996a, Douglas:1941a, anderson-duchamp:1984a, saunders:2010a, do-prince:1941a,krupka:1981a}.

The variational bicomplex \cite{anderson:2016a,anderson:1992a,tsujishita:1982a} can be used to provide an invariant solution to the inverse problem by utilizing the Helmholtz conditions.  In terms of the variational bicomplex, the existence of a solution to \ref{InvProb} can be expressed by the existence of an element of the cohomology space $H^{n-1,2}$ for the equation of a special algebraic nature where $n$ is the number of independent variables.  The existence of a non-trivial cohomology class for the equation $\Delta$ can then in principle be expressed in terms of invariants of the equation such as in \cite{anderson-thompson:1992a, fels:1996a}.

One of the goals of this article is to give a complete interpretation of the cohomology space $H^{1,2}(\CR)$ for scalar evolution equations $u_t=K(t, x, u, u_x,\ldots)$ which extends the interpretation of the special elements which control the solution to the inverse problem. The result is a natural generalization of the inverse problem in equation \ref{InvProb} we call the variational operator problem. Given a differential equation $\Delta=0$, does there exist a differential operator $\CE$ and Lagrangian $L$ such that
\begin{equation}
\CE(\Delta) = \Eop(L)
\label{GenInv}
\end{equation}
A simple example is given by the potential cylindrical KdV equation, $ u_t = u_{xxx}+\frac{1}{2}u_x^2-\frac{u}{2t}$ which admits $\CE=tD_x$ as a first order variational operator,
$$
tD_x \left( u_t - u_{xxx}-\frac{1}{2}u_x^2+\frac{u}{2t}\right)=
\Eop\left( -\frac{1}{2}t u_x u_t +\frac{1}{2}t u_x u_{xxx}+\frac{1}{6}tu_x^3 \right).
$$
The variational operator problem in equation  \ref{GenInv} can be studied for either the case of scalar or systems of ordinary or partial differential equations. Here we restrict our attention to problem \ref{GenInv} in the case where $\Delta$ is a scalar evolution equations in order to relate this problem to the theory of symplectic and Hamiltonian operators for integrable systems.  

\medskip
In Section \ref{prelim} we give a quick summary of the relevant facts about the variational bicomplex for the case
we need.  Sections \ref{CFNU} and \ref{CFU} provide normal forms for the cohomology spaces $H^{r,s}(\CR)$ 
in the variational bicomplex associated with the equation $\Delta=0$. These forms are then used in Section \ref{VOH12} to show there exists a one to one correspondence between the solution to \ref{GenInv} and the cohomology space $H^{1,2}(\CR)$. Even order evolution equations don't admit variational operators but we have the following theorem for odd order equations (the summation convention is assumed).

\begin{Theorem}  \label{THA}  Let $\CE = r_i(t, x, u, u_x,\ldots) D_x^i$ $i=0,\ldots, k$ be a $k^{th}$ order differential operator and let the zero set of $\Delta=u_t-K(t,x,u,u_x,\ldots, u_{2m+1}), m\geq 1$ define an odd order evolution equation.

\begin{enumerate}
\item The operator $\CE$ is a variational operator for $\Delta$ if and only if  $\CE$ is skew-adjoint and
\begin{equation}
\omega = dx \wedge \theta^0 \wedge \epsilon -dt \wedge \sum _{j=1}^{2m+1}  \left( \sum_{a=1}^{j} (-X)^{a-1}\left(\frac{\partial K}{\partial u_{j}} \epsilon\right)\wedge \theta^{j-a} \right)
\label{omegad}
\end{equation}
is $d_H$ closed on $\CR$, where  $\epsilon = -\frac{1}{2} r_i \theta^i$ and $\theta^i$ are given in equation \ref{DT}.

\item Let $\CV_{op}(\Delta)$ be the vector space of variational operators for $\Delta$. The function $\Phi:\CV_{op}(\Delta)\to H^{1,2}(\CR)$ defined from equation \ref{omegad} by
\begin{equation}
\Phi(\CE) =[\, \omega \, ],
\label{FISO}
\end{equation}
is an isomorphism.
\end{enumerate}
\end{Theorem}

The isomorphism property of $\Phi$ in Theorem \ref{THA} implies that a scalar evolution equation admits a variational operator if and only if  $ H^{1,2}(\CR)\neq 0$. Moreover the operator $\CE$ (and subsequently the function $L$) in \ref{GenInv} are easily determined from $[\omega]\in H^{1,2}(\CR)$  see Theorem \ref{conf_2}.  

Theorem \ref{THA} converts the solution to the operator problem \ref{GenInv} for a differential equation into a cohomology computation for the equation. The techniques developed for solving the multiplier inverse problem in terms of cohomology \cite{anderson-thompson:1992a, fels:1996a,anderson-duchamp:1984a} can then be used to solve the operator problem.

\medskip

A related problem  to \ref{GenInv} in the theory of integrable systems is the notion 
of a symplectic Hamiltonian evolution equation \cite{dorfman:1993a}  which is reviewed in Section \ref{FF2} in terms of the variational bicomplex. In the time independent case, a scalar evolution $u_t=K(x,u,u_x,\ldots,u_n)$ equation is said to be Hamiltonian with respect to a time independent symplectic operator $\CS= s_i( x, u, u_x,\ldots) D_x^i$ if 
\begin{equation}
\CS(K) = \Eop(L).
\label{SymOp}
\end{equation}
For a time dependent equation and  operator, condition \ref{SymOp} is given in \ref{HVF2t} in terms of the symplectic potential. Symplectic operators exists on a different space than variational operators but there is a natural identification (see Remark \ref{thm_rmk}) between symplectic operators and operators which can be variational operators. With  this identification, problems \ref{GenInv} and \ref{SymOp} are shown to be the same and in Section \ref{Symplectic} and we have the following theorem.

\begin{Theorem} \label{TIS2}  Let $\CS = s_i(t, x, u, u_x,\ldots) D_x^i$ be a differential operator and let $\Delta=u_t-K(t,x,u,u_x,\ldots,u_n)$. The operator $\CS$ is a symplectic operator and $\Delta=0$ is a symplectic Hamiltonian system for $\CS$ if and only if $\CS$ is a variational operator for $\Delta$.
\end{Theorem}

Theorem \ref{TIS2} shows that symplectic operators and variational operators for $u_t=K$ are the same so that Theorem \ref{THA} implies the following. 

\begin{Theorem}  \label{THC} The function $\Phi$ in equation \ref{FISO} defines an isomorphism between the vector space of symplectic operators  $\CS=\CE=r_i(t, x, u,u_x,\ldots) D_x^i$ for which $\Delta=u_t-K$ is Hamiltonian, and the cohomology space $H^{1,2}(\CR)$. 
\end{Theorem}

With Theorem \ref{THC} in hand, the determination of a symplectic Hamiltonian formulation of $u_t=K$ is resolvable in terms of the cohomology $H^{1,2}(\CR)$ of the differential equation $u_t=K$ and subsequently the invariants of $\Delta$.  This characterization of symplectic Hamiltonian evolution equations in terms of $H^{1,2}(\CR)$ allows the techniques in \cite{anderson-thompson:1992a, fels:1996a,anderson-duchamp:1984a} to be used in their study.  

A key idea that directly explains the interplay between the symplectic Hamiltonian formulation for an evolution equation and the cohomology $H^{1,2}(\CR)$ is the fact that the equation manifold $\CR$ is canonically diffeomorphic to $\reals \times J^\infty(\reals,\reals)$. The cohomology of the equation is expressed in terms 
of the geometric structure that arises from the embedding into $J^\infty(\reals^2,\reals)$ while the symplectic Hamiltonian formulation of an equation is expressed in terms of the contact structure on $\reals\times J^\infty(\reals,\reals)$. Theorem \ref{keyT}  shows how these are related and this leads to Theorem \ref{THC}.

\smallskip

In Section \ref{FirstOrder} the case of first order operators for third order equations are examined in detail and the following characterization is found.

\begin{Theorem} \label{FOT}  A third order scalar evolution equation $u_t=K(t, x, u , u_x, u_{xx},u_{xxx})$  admits a first order symplectic operator (or variational operator) $\CE=2RD_x+D_x R$ if and only if $\kappa$ is a trivial conservation law, where
\begin{equation}
\kappa=\hat K_2 \, dx+  \left( -2K_0 +K_1\hat K_2 -\frac{1}{2}\left(X(K_3)\hat K_2 ^2 +K_3 \hat K_2^3\right) +X\left(K_3X(\hat K_2)\right)\right) dt
\label{kappadef}
\end{equation}
and  $K_{i} = \partial_{i} K$,  $ \hat K_2 = \frac{2}{3K_{3}}(K_{2} - X(K_{3}))$, and
$X$ is the total $x$ derivative on $\CR$.

Furthermore, when  $\kappa= d_H (\log R) $  then $u_t=K$ admits the first order symplectic (or variational) operator $\CE=   2R D_x+D_xR $
\end{Theorem}

In Section \ref{FirstOrder} we examine the relationship between the Hamiltonian form of evolution equations and their potential form. In \cite{Nutku:2002a} it is shown that the (first order) potential form of a Hamiltonian equation admits a variational operator. We examine this in more detail, as well as the role of bi-Hamiltonian systems as in \cite{pavlov:2017a}. This theory is used in Example \ref{HD1} where the Krichever-Novikov equation (or Schwartzian KdV) is shown to be the potential form of the Harry-Dym equation. The symplectic operators (or variational operators) for the Krichever-Novikov equation arise as the lift of the Hamiltonian operators of the Harry-Dym equation as described in Section \ref{biht}.

Theorem \ref{FOT} should be contrasted to the problem of determining a Hamiltonian formulation of a scalar evolution equation in terms of a  Hamiltonian operator. An evolution equation $u_t=K$ is Hamiltonian with respect to a Hamiltonian operator $\CD$ if there exists  a Hamiltonian function $H$ (see \cite{vino:1986a,dorfman:1993a,Olver:1993a}) such that
\begin{equation}
u_t = \CD \circ \Eop(H)
\label{Ham}
\end{equation}
Conditions for the existence of $\CD$ and $H$ in equation \ref{Ham} in terms of the invariants of $u_t=K$ is unknown. We illustrate the difference in these problems with the cylindrical KdV and its potential form. The potential form of the cylindrical KdV is easily shown to admit at least two variational (or symplectic) operators. Section \ref{biht} then suggests that the cylindrical KdV is a bi-Hamiltonian system. See Example \ref{ECKDV} where a bi-Hamiltonian formulation of the cylindrical KdV is proposed (\cite{wang:2002a} states that no Hamiltonian exists for the cylindrical KdV).

Lastly, in Appendix \ref{SdV} we identify the elements of $H^{1,1}(\CR)$ which don't arise as the vertical differential of a conservation law with a family of variational operators. Example \ref{HD2} demonstrates the theory.


\section{Preliminaries}\label{prelim}

In this section we review some basic facts on the variational bicomplex associated with scalar evolution equations, see \cite{anderson-kamran:1997a} for more details.

\subsection {The Variational Bicomplex on $J^\infty(\reals^2,\reals)$}

The $t$ and $x$ total derivative vector fields on $J^\infty(\reals^2,\reals)$ with coordinates $(t,x,u,u_t,u_x,u_{tt},u_{tx},u_{xx},\ldots )$  are given by
$$
\begin{aligned}
D_t & = \partial_t + u_t \partial_u + u_{tt}\partial_{u_t} + u_{tx} \partial_{u_x} + \ldots\\
D_x & = \partial_x + u_x \partial_u  + u_{tx} \partial_{u_t} + u_{xx}\partial_{u_x} + \ldots \ .
\end{aligned}
$$
The contact forms on $J^\infty(\reals^2,\reals)$  are
\begin{equation}
\begin{aligned}
\vartheta^0 & = du -u_t dt -u_x dx\\
\vartheta^{i} & = D^i_x\left(du- u_t dt - u_x dx\right)= du_i -u_{t,i}dt-u_i dx, \quad i \geq 1 \\
\quad \zeta^{a,j} & = D^i_x D^a_t \left(du- u_t dt - u_x dx\right),   \quad a \geq 1, \ i \geq 0
\end{aligned}
\label{VTD}
\end{equation}
where $u_i = D_x^i(u)$ and $u_{a,i} = D^i_xD^a_t( u) = u_{tttt\ldots,xxx\ldots }$.  

The free variational bicomplex on $\CJ$ is denoted by $\Omega^{r,s}(\CJ)$ where $\omega \in \Omega^{r,s}(\CJ)$ is a differential form of degree $r+s$ which is horizontal of degree $r$ and vertical of degree $s$ (see Section 2 in \cite{anderson-kamran:1997a}). In particular if $\omega\in \Omega^{1,2}(\CJ)$ then
$$
\omega= dt \wedge \alpha + dx \wedge \beta,\quad
\alpha , \beta \in {\rm span} \{ \vartheta^i \wedge \vartheta^j, \vartheta^i \wedge \zeta^{a,j},  \zeta^{a,i}\wedge  \zeta^{b,j} \}.
$$
The horizontal and vertical differentials are anti-derivations
$$
d_\sH : \Omega^{r,s}(\EJ) \to \Omega^{r+1,s}(\EJ), \quad d_\sV : \Omega^{r,s}(\EJ) \to \Omega^{r,s+1}(\EJ)
$$
which satisfy 
$$
\begin{aligned}
d_\sH  f  =D_t(f) dt + D_x( f ) dx,\quad  
d_\sV f =\frac{\partial f}{\partial u_{i}} \vartheta^i+\frac{\partial f}{\partial u_{t,i}} \zeta^{1,i}+\ldots,\\
d_\sH \vartheta^i= -\zeta^{1,i} \wedge dt -\vartheta^{i+1} \wedge dx ,  \quad \tdV  \vartheta^I =0,
\end{aligned}
$$
where $f\in C^\infty(\CJ)$. Since $d= d_\sH+d_\sV$ this implies,
$$
d_\sH^2=0,\quad d_\sV^2=0,\quad {\rm and}\quad \tdH \tdV +\tdV\tdH = 0.
$$

The integration by parts operator $I: \Omega^{2,s}(J^\infty(\reals^2,\reals)) \to \Omega^{2,s}(J^\infty(\reals^2,\reals) )$ is defined by 
\begin{equation}
I(\omega)  = \frac{1}{s} \vartheta^0 \wedge \sum_{a=0,i=0}^\infty  (-1)^{i+a} D_x^iD_t^a( (\partial_{u_{a,i}} \hook \omega) ,
\label{IE2}
\end{equation}
and it has the following properties \cite{anderson:1992a}, \cite{anderson:2016a},
\begin{equation}
I^2=I, \quad
\omega= I(\omega) +\tdH \eta, \qquad {\rm for \ some \ } \eta \in \Omega^{1,s}(J^\infty(\reals^2,\reals) ).
\label{IEP2}
\end{equation}
If we let $J: \Omega^{r,s}(J^\infty(\reals^2,\reals) \to \Omega^{r-1,s}(J^\infty(\reals^2,\reals) )$ be
\begin{equation}
J(\kappa) = \sum_{a,i=0}^\infty  (-1)^{i+a} D_x^iD_t^a( (\partial_{u_{a,i}} \hook \kappa) 
\label{IEP2P}
\end{equation}
then $ I (\kappa) = \frac{1}{s} \vartheta^0 \wedge J(\kappa)$. Both $J$ and $I$ satisfy,
\begin{equation}
\ker J= \ker I = {\rm Im}\, d_\sH \ .
\label{IEP3}
\end{equation}
The operator $J$ is the interior Euler operator, see page 292 in \cite{anderson-kamran:1997a} or page 43 in \cite{anderson:2016a}.

Let $\CE= r_{ia} D_x^i D_t^a $ be a total differential operator. The formal adjoint  $ \CE^*$ 
is the total differential operator characterized as follows. For any $ \rho \in \Omega^{0,s}(\CJ)$ and $ \omega \in \Omega^{0,s'}(\CJ)$ there exists $ \zeta \in \Omega^{1,s+s'}(\CJ)$ depending on $\rho$ and $\omega$ such that
\begin{equation}
( \rho \wedge  \CE(\omega) - \CE^*( \rho)\wedge \omega) \wedge dt \wedge dx  = d_H \tilde \zeta.
\label{d_ad}
\end{equation}
This leads to 
$$
\CE^*(\alpha) = (-1)^{i+a}D_x^iD_t^a( r_{ai} \alpha) \ , \quad \alpha \in \Omega^{r,s}(\CJ).
$$
It follows from \ref{d_ad} that the formal adjoint satisfies $(\CE^*)^*=\CE$.

Let $\Delta $ be a smooth function on $J^\infty(\reals^2,\reals)$. The Fr\'echet derivative of $\Delta$ \cite{Olver:1993a} is the total differential operator ${\bf F}_\Delta$  satisfying $ d_\sV \Delta = {\bf F}_\Delta(\vartheta^0) $. If
 $$
\Delta= u_t - K(t,x,u,u_x,\ldots, u_n),
$$
then
\begin{equation}
d_\sV \Delta = \vartheta_t - K_i \vartheta^i   \qquad {\rm where} \quad 
K_i = \frac{\partial \ } {\partial u_i}   K(t,x,u,u_x,\ldots, u_n), \quad  i = 0,\ldots, n.
\label{dVF}
\end{equation}
The Fr\'echet derivative of  $\Delta$ is determined from equation \ref{dVF} to be the total differential operator
\begin{equation}
{\bf F}_\Delta= D_t -\sum_{i=0}^n  K_i D_x^i \ .
\label{CL_d}
\end{equation}
The adjoint of the operator  in \ref{CL_d} is,
$$
{\bf F}_\Delta ^*( \rho )= -D_t (\rho)-\sum_{i=0}^n (-D_x)^i ( K_i  \rho), \quad \rho \in \Omega^*(J^\infty(\reals^2,\reals).
$$

\subsection {The Variational Bicomplex on $\CR$ and $H^{r,s}(\CR)$}

Let $\Delta=u_t - K(t,x,u,u_x,\ldots, u_n)$ and let
$\CR$ be the infinite dimensional manifold which is the zero set of the prolongation of $\Delta=0$ 
in $J^\infty(\reals^2,\reals)$. With coordinates $(t,x,u,u_x,u_{xx}, \ldots )$ on $\CR$
the embedding  $\iota : \CR \to J^\infty(\reals^2,\reals)$ is given by
\begin{equation}
\iota = [t=t,x=x,u=u,u_t=K, u_x=u_x, u_{tt}= T(K) , u_{tx} = X(K), u_{xx} = u_{xx}, \ldots ],
\label{defiota}
\end{equation}
where $T$ and $X$ are the restriction of $D_t$ and $D_x$ to $\CR$ given by,
\begin{equation}
\begin{aligned}
X& = \partial_x + u_x \partial_u + u_{xx} \partial_{u_x}+\ldots \\
T& = \partial_t + K \partial_u + X(K) \partial_{u_x} + \ldots 
\end{aligned}
\label{DTX}
\end{equation}
and satisfy $[X,T]=0$. The Pfaffian system  $\CI= \{ \theta^ i\} _{i\geq 0}$ on $\CR$ is generated by the pullback of $\vartheta^i$ in equation \ref{VTD}
\begin{equation}
\theta^i= \iota^* \vartheta^i =  du_i-X^i(K) dt -u_{i+1} dx .
\label{DT}
\end{equation}
  
The forms
\begin{equation}
\{ dt,\   dx,  \theta^i = du^i - X^i(K) dt - u_{i+1} dx \}\quad i=0,1,\ldots
\label{CoF}
\end{equation}
form a coframe on $\CR$, and give rise to a vertical and horizontal splitting in 
the complex of differential forms leading to the bicomplex $\Omega^{r,s}(\CR)$, $r=0,1,2$ and $s=0,1,\ldots$.  For example if $\omega \in \Omega^{1,2}(\CR)$ then 
$$
\omega = dx \wedge \alpha + dy \wedge \beta
$$
where $\alpha = a_{ij} \theta^i\wedge \theta^j$ and $\beta = b_{ij} \theta^i\wedge \theta^j$, $a_{ij}, b_{ij} \in C^\infty(\CR)$. The bicomplex $\Omega^{r,s}(\CR)$ is the pullback of the free bicomplex $\Omega^{r,s}(\CJ)$ by the embedding $\iota:\CR \to \CJ$.

The horizontal exterior derivative $d_\sH:\Omega^{r,s}(\CR)\to  \Omega^{r+1,s}(\CR)$ and vertical exterior derivative $d_\sV:\Omega^{r,s}(\CR)\to  \Omega^{r,s+1}(\CR)$ are computed from the equations,
\begin{equation}
d_\sH(\omega) = dx\wedge X(\omega) + dt \wedge T(\omega), \quad d_\sV= d- d_\sH.
\label{dHd}
\end{equation}
The horizontal and vertical differentials satisfy
\begin{equation}
d^2_\sH = 0 \qquad d^2_\sV =0, \qquad d_\sH d_\sV = - d_\sV d_\sH.
\label{dHP2}
\end{equation}
The structure equations of $\CI$ are computed using \ref{DT} to be
\begin{equation}
d_\sH \theta^i = dx \wedge \theta^{i+1}  + dt \wedge X^i(d_\sV K )  \quad  {\rm  and} \quad  d_\sV \theta^i = 0.
\label{SE_c}
\end{equation}

Since $d_\sH^2=0$, the complex $d_\sH :\Omega^{r,s}(\CR)\to \Omega^{r+1,s}(\CR)$ is a differential complex  and $H^{r,s}(\CR)$ is its cohomology,
$$
H^{r,s}(\CR) = \frac{ \ker \{d_\sH :\Omega^{r,s}(\CR) \to \Omega^{r+1,s}(\CR) \} }{ {\rm Im} 
\{d_\sH:\Omega^{r-1,s}(\CR) \to \Omega^{r,s}(\CR) \}  }.
$$
The conservation laws of $\Delta$ are the $d_\sH$ closed forms in $ \Omega^{1,0}(\CR)$ and $H^{1,0}(\CR)$ is the space of equivalence classes of conservation laws modulo the horizontal derivative of a function $d_\sH f$, $f\in C^\infty(\CR)$.

The vertical complex $d_\sV :\Omega^{r,s}(\CR) \to \Omega^{r,s+1}(\CR)$ is a differential complex whose cohomology is trivial \cite{anderson:2016a}, \cite{anderson-kamran:1997a}.  Specifically, $d_\sV$ is the ordinary exterior derivative in the variables $u_i$, and the DeRham homotopy formula (in $u_i$ variables with  parameter) applies. The property $d_\sH d_\sV =-d_\sV d_\sH$ make $d_\sV : H^{r,s} (\CR) \to H^{r,s+1}(\CR)$ a co-chain map up to sign, see Appendix \ref{SdV}.

\begin{Remark} \label{thm_rmk}  Every function $Q(t, x, u, u_x, u_{xx}, \ldots,u_k )$ on $J^\infty(\reals^2,\reals)$ factors through $\pi:J^\infty(\reals^2,\reals) \to \CR$, 
$\pi(t, x, u, u_t, u_x,u_{tt},u_{tx},u_{xx},\ldots)=(t, x, u, u_x,u_{xx}, \ldots)$ which is a left inverse of $\iota$ in equation \ref{defiota}.   Therefore, by an abuse of notation, we view a function $Q(t, x, u, u_x, u_{xx}, \ldots,u_k )$ either on $J^\infty(\reals^2,\reals)$ or $\CR$ where the context will determine which. For example,
$$
{\bf F}_ Q = Q_i D_x^i
$$
is a differential operator on $J^\infty(\reals^2,\reals)$ while
$$
{\bf L}_Q = Q_i X^i
$$
is a differential operator on $\CR$ where $\pi_*({\bf F}_ Q )= {\bf L}_ Q$ .  
A differential operator $\bar \CE = r_i(t, x, u, u_x,\ldots) X^i$ on $\CR $ lifts to $\CE= r_i D_x^i$ and $\pi_* \CE = \bar \CE$. The formal adjoint of $\bar \CE$ acting on a form $\omega$ is $(-X_i)^i(r_i \omega)$. The operator $\bar \CE$ is skew-adjoint if and only if $\CE$ is skew adjoint.
\end{Remark}

\section{ Canonical forms for $H^{1,s}(\CR) $ and characteristic forms}\label{CFNU}

The universal linearization (see \cite{anderson-kamran:1997a}) of $\Delta=u_t-K(t, x, u, u_x,\ldots, u_n)$ on $\CR$ is the differential operator (on $\CR$),
\begin{equation}
{\CL}_\Delta= T-\sum_{i=0}^n K_i X^i
\label{Ad_L}
\end{equation}
where $K_i=\partial_{u_i} K$, and the vector fields $T$ and $X$ are defined in equation \ref{DTX}. The operator $\CL_\Delta$ is the restriction of the Fr\'echet derivative of $\Delta$ to $\CR$. The adjoint of $\CL_\Delta$ is given by 
$$ 
{\CL}^*_\Delta(\rho) = -T(\rho) - \sum_{i=0}^n (- X)^i (K_i \rho), \qquad \rho\in \Omega^*(\CR).
$$

This next theorem provides a normal form for a representative 
of the cohomology classes in $H^{1,s}(\CR)$ and is analogous to Theorem 5.1 in \cite{anderson-kamran:1997a}.

\begin{Theorem} \label{Thmcf1} Let $u_t=K(t, x, u, u_x,\ldots, u_n)$ be an $n^{th}$ order evolution equation and $H^{r,s}(\CR)$ its cohomology. For any $[\omega] \in H^{1,s}(\CR)$ there exists a representative,
\begin{equation}
\begin{aligned}
\omega= dx \wedge \theta^0 \wedge \rho   - dt \wedge \beta, \qquad 
\end{aligned}
\label{CF1}
\end{equation}
where $\rho \in \Omega^{0,s-1}({\mathcal R}^\infty), \ 
\beta \in \Omega^{0,s}({\mathcal R}^\infty)$ and  ${\CL}_\Delta^*(\rho) =0$.
\end{Theorem}

\begin{proof} The proof follows Theorem 5.1 of \cite{anderson-kamran:1997a}. Choose $\tilde \omega_0 \in \Omega^{1,s} (J^\infty(\reals^2,\reals))$ such that
\begin{equation}
\iota^*(\tilde \omega_0) \in [ \omega].
\label{pbw0}
\end{equation}
where $\iota:\CR\to J^\infty(\reals^2,\reals)$ is given in equation \ref{defiota}.
Since $\iota^*( d_\sH \tilde \omega_0 ) =0$, it there exists $\tilde \zeta_{ab} \in \Omega^{0,s} (J^\infty(\reals^2,\reals)), \tilde \mu_{ab} \in \Omega^{0,s-1} (J^\infty(\reals^2,\reals))$ such that (see Lemma 5.2 in \cite{anderson-kamran:1997a})
\begin{equation}
d_\sH \tilde \omega_0 = dt \wedge dx \wedge (D_t^aD_x^b(\Delta)  \tilde \zeta_{ab} + D_t^aD_x^b (d_\sV \Delta) \wedge \tilde \mu_{ab}),
\label{r_hot0}
\end{equation}
Applying the identical integration by parts argument on page 292 \cite{anderson-kamran:1997a}  to \ref{r_hot0}, implies there exists $ \tilde \zeta \in \Omega^{0,s} (J^\infty(\reals^2,\reals)), \tilde \rho \in \Omega^{0,s-1} (J^\infty(\reals^2,\reals))$ and $\tilde \omega\in \Omega^{1,s}(J^\infty(\reals^2,\reals))$ such that $ \tilde \omega = \tilde \omega_0 + \tilde d_\sH \tilde \eta$ and $\iota^* \tilde \eta =0$ (hence $\iota^* \tilde \omega =\omega$) and where
\begin{equation}
d_\sH \tilde \omega = dt \wedge dx \wedge (\Delta  \tilde \zeta + d_\sV \Delta \wedge \tilde \rho).
\label{r_hot}
\end{equation}

We now apply $\iota^* \circ J $ to equation  \ref{r_hot},  where $ J $ is defined in equation \ref{IEP2P}.
For the  first term in right hand side of equation \ref{r_hot} we find
\begin{equation}
\iota^*\circ J(\Delta \tilde \zeta)= \iota^*\left( \Delta \partial_u \hook( \tilde \zeta )
 -D_t (   \Delta  \partial_{u_t} \hook   \tilde \zeta)
 -D_x ( \Delta \partial_{u_x} \hook \tilde \zeta) +\ldots \right) =0
\label{JE1}
\end{equation}
since each term contains a total derivative of $\Delta$, and these vanish under pullback to $\CR$.

We now apply $\iota^*\circ J$ to the second term in the right hand side of \ref{r_hot},
\begin{equation}
\begin{aligned}
\iota^*J(d_\sV \Delta \wedge \tilde \rho) &= \iota^*\left(
\sum_{a=0,i=0}^\infty  (-1)^{i+a} D_x^iD_t^a( (\partial_{u_{a,i}} \hook d_\sV\Delta) \tilde \rho -(\partial_{u_{a,i}} \hook \tilde \rho) d_\sV\Delta) \right) \\
&= \iota^*\left(
\sum_{a=0,i=0}^\infty  (-1)^{i+a} D_x^iD_t^a( (\partial_{u_{a,i}} \hook d_\sV\Delta) \tilde \rho)  \right) 
\end{aligned}
\label{PA1}
\end{equation}
because $\iota^* ( D_t^aD_x^i d_\sV\Delta) = 0$. Now 
$$
\partial_{u_{1,0}} \hook d_\sV\Delta=1, \quad \partial_{u_{0,i}} \hook d_\sV\Delta=-K_i
$$
with all other $\partial_{u_{a,i}} \hook d_\sV\Delta=0$,  so that equation \ref{PA1} becomes,
\begin{equation}
\iota^*J(d_\sV \Delta \wedge \tilde \rho)=\iota^*( -D_t( \tilde \rho) + \sum_{i=0}^n (-1)^i D_x^i( -K_i \tilde \rho))= \CL_\Delta^*( \iota^* \tilde \rho).
\label{FIJ2}
\end{equation}
By equation \ref{IEP3} $J (d_\sH \tilde \omega) = 0$, so that applying $\iota^* \circ J$ to  equation \ref{r_hot} implies $\iota^*J(d_\sV \Delta \wedge \tilde \rho)=0$, and so  equation \ref{FIJ2} gives $\CL_\Delta^*(\iota^* \tilde \rho)=0$.

We now turn to showing that equation \ref{CF1} holds using the horizontal homotopy operator (equations 5.15, 5.16 and below 5.16 in \cite{anderson-kamran:1997a}), see also proposition 4.12 page 117 of \cite{anderson:2016a} or equation 5.133 in \cite{Olver:1993a}.  Using the notation $h^{r,s}_\sH$ from \cite{anderson:2016a}, this operator satisfies  
$$
\tilde \omega = h_\sH^{2,s}(d_\sH \tilde \omega) +d_\sH( h_\sH^{1,s} \tilde \omega ), \quad \tilde \omega \in \Omega^{1,s}(J^\infty(\reals^2,\reals)).
$$ 
Applying the pull back  by $\iota$ to this formula gives the representative for $[\omega]$,
\begin{equation}
\omega = \iota^* h_\sH^{2,s}(d_\sH \tilde \omega)  
\label{PBW1}
\end{equation}
with $d_\sH \tilde \omega $ in \ref{r_hot}.

To utilize the formula in \cite{anderson:2016a} for $h^{2,s}_\sH$  let $(x^1=t, x^2=x)$ and so for example $\vartheta^{1122}=D_x D_x D_t D_t \vartheta^0$ and let $k$ be the max of $ |I| $ (number of derivatives) of $\vartheta^I$ terms in $ (\vartheta_t -K_i \vartheta^i)\wedge \rho$.  Then by definition 4.13 on page 117 in \cite{anderson:2016a} (or 5.134 in \cite{Olver:1993a}) 

\begin{equation}
\begin{aligned}
h_\sH^{2,s}(d_\sH \tilde \omega) & =
\frac{1}{s} \sum_{|I|=0}^{k-1} D_I\left( \vartheta^0 \wedge J^{Ij} ((d_\sH \tilde \omega)_j)\right)\\
& =\sum_{|I|=0}^{k-1} D_I\left( \vartheta^0 \wedge  (-1)^{j} d\hat x_j \wedge J^{Ij}(\Delta  \tilde \zeta+d_\sV \Delta \wedge \tilde \rho)\right)
\end{aligned}
\label{HFH}
\end{equation}
where $(d_\sH \tilde \omega)_j  = D_{x^j} \hook  d_\sH \tilde \omega=   (-1)^{j+1} d\hat x_j \wedge (\Delta  \tilde \zeta+ d_\sV \Delta \wedge \tilde \rho)$,  $(\hat x_1=x, \hat x_2=t)$,   $I =(i_1,\ldots, i_l)$, $|I|=l$, and  
\begin{equation}
J^{Ij}(\Delta  \tilde \zeta+d_\sV \Delta \wedge \tilde \rho) = \sum _{|L|=0}^{k-|I|-1} { { |I|+|L|+1} \choose{ |L|}}(-D)_L ( \partial_{u_{(jIL)}} \hook (\Delta  \tilde \zeta+d_\sV \Delta \wedge \tilde \rho)).
\label{defJIj}
\end{equation}

Applying $\iota^*$ to equation \ref{HFH} we have the $\iota^* h^{2,s}_\sH\left(dt\wedge dx\wedge ( \Delta  \tilde \zeta )\right)=0$ because all terms in \ref{defJIj} on $\Delta  \tilde \zeta $ involve total derivatives of $\Delta$. Therefore using equation \ref{HFH}, equation \ref{PBW1} becomes,
\begin{equation}
\begin{aligned}
\omega =\frac{1}{s} \iota^* \bigg( \!\!\!\!\!\!\!  & & dx \wedge
\sum_{|I|=0}^{k-1} D_I\left[  \vartheta^0  \wedge 
\sum _{|L|=0}^{k-|I|-1} { { |I|+|L|+1} \choose{ |L|}}(-D)_L ( (\partial_{u_{(1IL)}} \hook d_\sV \Delta)\tilde \rho ) \right] \quad \ \  \\
& &\  -dt \wedge
 \sum_{|I|=0}^{k-1} D_I\left[    \vartheta^0  \wedge 
\sum _{|L|=0}^{k-|I|-1} { { |I|+|L|+1} \choose{ |L|}}(-D)_L ( (\partial_{u_{(2IL)}} \hook d_\sV \Delta) \tilde\rho ) \right] \quad  \bigg)
\end{aligned}
\label{BP1}
\end{equation}

 


Consider the first term in equation \ref{BP1}. The only non-zero interior product is (with $u_{(1)}=u_t, u_{(2)}=u_x$ etc.)
$$
\partial_{u_{(1)}}\hook d_\sV \Delta = 1, 
$$
since $\Delta$ does not depend on derivatives such as $u_{tx}$.
Therefore the only non-zero terms have $|I|=0, |L|=0$ in the first term of \ref{BP1} giving,
\begin{equation}
dx \wedge \sum_{|I|=0}^{k-1} D_I\left[  \vartheta^0  \wedge 
\sum _{|L|=0}^{k-|I|-1} { { |I|+|L|+1} \choose{ |L|}}(-D)_L ( (\partial_{u_{(1IL)}} \hook d_\sV \Delta)\tilde \rho )\right] = dx \wedge \vartheta^0 \wedge \tilde \rho
\label{AD1}
\end{equation}
Combining equation \ref{AD1} with \ref{BP1} we have
\begin{equation}
\omega= \frac{1}{s} \iota^*\bigg( dx \wedge \vartheta^0 \wedge \tilde \rho 
-dt \wedge
 \sum_{|I|=0}^{k-1} D_I\left[    \vartheta^0  \wedge 
\sum _{|L|=0}^{k-|I|-1} { { |I|+|L|+1} \choose{ |L|}}(-D)_L ( (\partial_{u_{(2IL)}} \hook d_\sV \Delta) \tilde\rho \right] \bigg) 
\label{endom}
\end{equation}
which produces equation \ref{CF1} with $\rho =\frac{1}{s} \iota^* \tilde \rho$ . Equations \ref{endom} and \ref{FIJ2} shows that $\rho =\frac{1}{s} \iota^* \tilde \rho$ satisfies $\CL_\Delta^*(\rho) = 0$.
\end{proof}

If $s=1$  in Theorem \ref{Thmcf1}, then $\rho\in C^\infty(\CR)$ and is the characteristic function for the cohomology class $[\omega]\in H^{1,1}(\CR)$, see Theorem \ref{Thmcf2} and Theorem \ref{CLisdV}. In general $\rho$ in equation \ref{CF1} is called a characteristic form for $[\omega]$ see \cite{anderson-kamran:1997a}. The form $\beta$ in \ref{CF1} is given in terms of $\rho$ by formula \ref{endom} which is simplified in Theorem \ref{Thmcf2} for $H^{1,1}(\CR)$ and $H^{1,2}(\CR)$. The term $dx \wedge \theta^0 \wedge \rho$ in equation \ref{CF1} generalizes the conserved density of a conservation law, and plays a critical role in Section \ref{Symplectic}.


\begin{Theorem} \label{ZC} For $s\geq 3$,  $H^{1,s} (\CR) = 0$.
\end{Theorem} 
\begin{proof}

Suppose $\omega$ is a representative for an 
element of $H^{1,s+1}(\CR)$, $(s\geq 2)$, in the form \ref{CF1}, where 
$\rho \in \Omega^{0,s}(\CR)$ is given by
\begin{equation}
\rho = A_{i_1\ldots  i_s} \theta^{i_1}\wedge \ldots \theta^{i_s}.
\label{rho2p}
\end{equation}
and satisfies $\CL_\Delta^*(\rho) =0$.

Suppose for $\rho$ in \ref{rho2p} that the highest form order for (no sum) $A_{m_1\ldots m_s} \theta^{m_1}\wedge\theta^{m_2}\wedge \ldots  \theta^{m_s} $ where we use max of ($m_1> m_2>\ldots>m_s$) to determine highest order. We first claim that in $ \CL_\Delta^*(\rho) $ the  coefficient of $\theta^{m_1+n} \wedge \theta^{m_2} \wedge \ldots \theta^{m_s}$  is
\begin{equation}
[\CL_\Delta^*(\rho)]_{m_1+n,m_2,\ldots,m_s} = (-1-(-1)^n)  K_n A_{m_1\ldots m_s} ,
\label{AHO0}
\end{equation}
and that the coefficient of $\theta^{m_1+n-1} \wedge \theta^{m_1} \wedge \ldots \theta^{m_s}$ when $m_1>m_2+1$ is
\begin{equation}
[\CL_\Delta^*(\rho)]_{m_1+n-1,m_2+1,\ldots,m_s} =
(-1)^n n K_n A_{m_1\ldots m_s}
\label{AHO}
\end{equation}
and when $m_1=m_2+1$ and $n$ is odd, the coefficient of $\theta^{m_1+n-1} \wedge \theta^{m_1} \wedge \ldots \theta^{m_s}$ 
\begin{equation}
[\CL_\Delta^*(\rho)]_{m_1+n-1,m_1,\ldots,m_s} =
-(-1)^n n K_n A_{m_1\ldots m_s}.
\label{AHO2}
\end{equation}
Therefore if $\CL_\Delta^*(\rho)=0$ implies $A_{m_1\ldots m_s} =0$, then $\rho=0$ and  $ \omega = dt \wedge \beta $. The condition $d_\sH \omega =0$ gives $X(\beta)  = 0$. This implies $\beta=0$ since $\beta \in \Omega^{0,s+1}(\CR)$ and so $\omega =0$.

We compute $ \CL_\Delta^*(\rho) = -T( \rho) -(-X)^i(K_i \rho)$ to find equations \ref{AHO0},\ref{AHO},\ref{AHO2},
\begin{equation}
\begin{aligned}
{\CL_\Delta^*}(\rho) = & -T(A_{i_1\ldots  i_s}) \theta^i_1 \ldots \wedge \theta^{i_s} -A_{i_1\ldots  i_s} T(\theta^{i_1})\ldots \wedge \theta^{i_s}-A_{i_1\ldots  i_s} \theta^{i_1}\wedge T(\theta^{i_2}) \ldots \wedge \theta^{i_s}+\ldots \\
&-(-1)^n X^n(K_nA_{i_1\ldots  i_s} \theta^{i_1}\wedge \ldots \theta^{i_s}) 
- (-1)^{n-1} X^{n-1}(K_{n-1}A_{i_1\ldots  i_s} \theta^i_1\wedge \ldots \theta^{i_s}) \ldots \label{CLrhol}
\end{aligned}
\end{equation}
where from equation \ref{SE_c} we have $T(\theta^i) = K_n\theta^{i+n}+{\rm lower \ order}$. Consider also the highest order terms in while expanding $X^n(K_nA_{i_1\ldots  i_s} \theta^{i_1}\wedge \ldots \theta^{i_s}) $,
\begin{equation}
\begin{aligned}
& \qquad \qquad \qquad X^n(K_nA_{m_1\ldots  m_s} \theta^{m_1}\wedge \theta^{m_2} \ldots \wedge \theta^{m_s})  \\
&=X^n(K_nA_{m_1\ldots  m_s} ) \theta^{m_1}\wedge \theta^{m_2}  \ldots \wedge \theta^{m_s}\\ 
&\quad + K_nA_{m_1\ldots  m_s}  \theta^{m_1+n} \wedge \theta^{m_2} \ldots \wedge \theta^{m_s} +nK_n
A_{m_1\ldots  m_s} \theta^{m_1+n-1}\wedge \theta^{m_2+1}  \ldots \wedge \theta^{m_s}+\ldots
\\
&  \quad + K_nA_{m_1\ldots  m_s} \theta^{m_1} \wedge \theta^{m_2+n} \ldots \wedge \theta^{m_s} + \ldots \\
& =  K_nA_{m_1\ldots  m_s}  \theta^{m_1+n} \wedge \theta^{m_2}  \ldots \wedge \theta^{m_s} +nK_n
A_{m_1\ldots  m_s}\theta^{m_1+n-1}\wedge \theta^{m_2+1}  \ldots\wedge \theta^{m_s}\\  &\quad  -
K_n A_{m_1\ldots  m_s} \theta^{m_2+n} \wedge \theta^{m_1} \ldots\wedge \theta^{m_s}+\{\rm \ essentially \ lower \ order \}.
\ldots 
\end{aligned}
\label{Xnrho}
\end{equation}

The coefficient of $\theta^{m_1+n} \wedge \theta^{m_2}\wedge \ldots  \theta^{m_s}$ (which is the highest order) occurring in equation \ref{CLrhol} comes from the second term on the right hand side in \ref{CLrhol} and the first term on the last right hand side in equation \ref{Xnrho} to give \ref{AHO0}.

We consider the next highest order term in \ref{CLrhol}. From equation \ref{CLrhol}, the only possible term that can contain  $\theta^{m_1+n-1} \wedge \theta^{m_2+1}\wedge \theta^{m_3}\ldots \theta^{m_s}$ when $m_1>m_2+1$  is from second term on the last right hand side of equation \ref{Xnrho}. Therefore \ref{Xnrho} produces \ref{AHO}.

In the case when $m_1=m_2+1$ we have the third term in \ref{CLrhol} at highest order giving
\begin{equation}
-A_{m_1\ldots  m_s} \theta^{m_1}\wedge T(\theta^{m_2}) \ldots \wedge \theta^{m_s}=
-A_{m_1\ldots  m_s} \theta^{m_1}\wedge (K_n\theta^{m_2+n})\ldots \wedge \theta^{m_s}+{\rm lower \ order},
\end{equation}
then using $m_1=m_2+1$ the first term equals
\begin{equation}
K_n A_{m_1\ldots  m_s}  \theta^{m_1-1+n}\wedge \theta^{m_1} \ldots \wedge \theta^{m_s}.
\label{shot1}
\end{equation}
From the other two terms on the last right hand side in equation \ref{Xnrho} we have the two terms
\begin{equation}
n A_{m_1\ldots  m_s}\theta^{m_1+n-1}\wedge \theta^{m_2+1}\ldots \wedge\ldots \theta^{m_s}-
K_n A_{m_1\ldots  m_s} \theta^{m_1+n-1}\wedge \theta^{m_2+1}  \ldots \wedge \theta^{m_s}.
\label{shot2}
\end{equation}
Combining equations \ref{shot1} and using $m_1=m_2+1$ and $n$ is odd in equation \ref{shot2}, gives equation \ref{AHO2}.
\end{proof}

We also have as a corollary of Theorem \ref{Thmcf1}.

\begin{Corollary}  If $u_t=K(t,x,u,\ldots, u_{2m})$ is an even order evolution equation, then $H^{1,2}(\CR)=0$. 
\end{Corollary}
\begin{proof} We show that the only $\rho \in \Omega^{0,1}(\CR)$ satisfying $\CL_\Delta^*(\rho) =0$ is $\rho =0$.  Suppose $\rho=r_i\theta^i,\ i=0,\ldots, k $ then by direct computation, 
$$
\CL_\Delta^*(\rho) =-T(\rho)-(-X)^i(K_i \rho)  \equiv -2 \, r_k K_{2m} \theta^{2m+k}  \qquad \mod \theta^0,\ldots, \theta^{2m+k-1}
$$
which is non-zero unless $r_k=0$. Therefore $\rho = 0$. 
\end{proof}

We now refine Theorem \ref{Thmcf1} which gives a formula for $\beta$ in equation \ref{CF1}.

\begin{Theorem} \label{Thmcf2} Let $u_t=K(t, x, u, u_x,\ldots, u_n)$ be an $n^{th}$ order evolution equation. For any $[\omega] \in H^{1,s}(\CR)$, $s=1,2$  there exists a representative,
$\omega  \in \Omega^{1,s}({\mathcal R}^\infty)$, $s=1,2$ such that
\begin{equation}
\omega = dx \wedge \theta^0 \wedge \rho   - dt \wedge \bbeta(\rho) , \quad
\bbeta(\rho)=  \sum _{i=1}^n  \left( \sum_{a=1}^{i} (-X)^{a-1}(K_{i} \rho)\wedge \theta^{i-a} \right),
\label{CF2}
\end{equation}
where $\rho \in \Omega^{0,s-1}({\mathcal R}^\infty)$ ($s=1,2$) satisfies $\theta^0 \wedge \CL_\Delta^*(\rho)=0$. If $s=1$ then $ \rho \in C^\infty(\CR)$ and the representative \ref{CF2} is unique.
\end{Theorem}

\begin{proof}  First suppose $\rho \in \Omega^{0,s-1}(\CR)$ and satisfies $\theta^0 \wedge \CL_\Delta^*(\rho)=0$, and let $\omega \in \Omega^{1,s}(\CR)$ be as in equation \ref{CF2}. We compute $d_\sH \omega$,
\begin{equation}
d_\sH \omega= dt \wedge dx \wedge [ T(\theta^0 \wedge \rho) + X(\bbeta(\rho))].
\label{dHB1}
\end{equation}
To compute $X(\bbeta(\rho))$ we need the telescoping identity,
\begin{equation}
X\left(\sum_{a=1}^{i} (-X)^{a-1}(K_i \rho)\wedge\theta^{i-a}\right) =- (-X)^i(K_i \rho)\wedge \theta^0 + K_i\rho \wedge \theta^i
\qquad ({\rm \ no \ sum \ on } \ i ).
\label{XB1}
\end{equation}
Using equation \ref{XB1} in the formula for $\bbeta(\rho)$ in equation \ref{CF2} gives
\begin{equation}
\begin{aligned}
X(\bbeta(\rho)) & =  \sum _{i=1}^n X\left[ \left( \sum_{a=1}^{i} (-X)^{a-1}(K_{i} \rho)\wedge \theta^{i-a} \right) \right], \\
&= \sum _{i=1}^n -(-X)^i(K_i \rho) \wedge \theta^0 + K_i\rho \wedge \theta^i .
\end{aligned}
\label{XB2}
\end{equation}
We then use
$$
T(\theta^0 \wedge \rho ) = T(\theta^0) \wedge \rho + \theta^0 \wedge  T(\rho)
$$
so that together with  equation \ref{XB2},  equation \ref{dHB1} becomes (adding and subtracting $K_0 \theta^0 \wedge \rho$)
\begin{equation}
d_\sH \omega  = dt \wedge dx \wedge  (T(\theta^0) - \sum_{i=0}^n K_i \theta^i) \wedge \rho
+dt \wedge dx \wedge \theta^0 \wedge \left( T(\rho) +  \sum _{i=0}^n (-X)^i(K_i \rho)\right)
\label{rho_d}
\end{equation}
Now  by equation \ref{dVF}, $\iota^* d_\sV(- \Delta)=-T(\theta^0) + \sum_{i=0}^n K_i \theta^i= 0 $, 
and so equation \ref{rho_d} becomes,
$$
d_\sH \omega  = dt \wedge dx \wedge \theta^0 \wedge \left( T(\rho) +  \sum _{i=0}^n (-X)^i(K_i \rho)\right)
=- dt \wedge dx \wedge \theta^0 \wedge \CL_\Delta^*(\rho) =0.
$$

Now  by  Theorem \ref{Thmcf2} there exists representative
$ \hat \omega = dx \wedge \theta^0 \wedge \rho - dt \wedge \beta$
where $\CL^*(\rho) = 0$. Let $\omega $ be the form in \ref{CF2} using this $\rho$.
The form $\omega' = \hat \omega - \omega$ satisfies 
\begin{equation}
0= d_\sH \omega'= d_\sH(  dt \wedge (\beta - \bbeta(\rho)) = dx \wedge dt \wedge X(\beta -\bbeta(\rho)) .
\label{dDB}
\end{equation}
This implies $X( \beta-\bbeta(\rho)) = 0$, where $\beta-\bbeta(\rho) \in \Omega^{0,s}(\CR), s=1,2$. However, the only contact form satisfying this condition is the zero form. So $\beta = \bbeta(\rho)$.  This proves equation \ref{CF2}.

For the final statement in the theorem, suppose  $\omega_a = dx \wedge \theta^0 \cdot Q_a - dt \wedge \beta_a$, $a=1,2$ where $Q_a \in C^\infty(\CR)$ and $\beta_a \in \Omega^{0,1}(\CR)$ satisfy $[\omega_1]=[\omega_2]\in H^{1,1}(\CR)$. This implies there exists $ \xi = g_j \theta^j$ such that $\omega_1-\omega_2 = d_\sH \xi$ so that 
\begin{equation}
dx \wedge \theta^0 (Q_1-Q_2) = dx \wedge X( g_j \theta^j).
\label{Us1}
\end{equation}
Since $X(\theta^i)=\theta^{i+1}$, equation \ref{Us1} can only be satisfied when $Q_1=Q_2$ and $g_j=0$.  Then the condition $d_\sH( \omega_1-\omega_2) = dx \wedge dt \wedge X( \beta_1-\beta_2) = 0$ implies $X(\beta_1-\beta_2)=0$ which in turn implies $\beta_1=\beta_2$ because  $\beta_1-\beta_2 \in \Omega^{0,1}(\CR)$.
Therefore the form in equation \ref{CF2} for $s=1$ is unique.
\end{proof}

The form $\omega$ in \ref{CF2} was originally derived by a rather lengthy calculation of the second term in \ref{endom}.

\begin{Corollary} \label{CFbeta} If $[\omega]\in H^{1,2}(\CR)$ with $\epsilon = r_i \theta^i,\ i=0,\dots ,k$ and representative 
$$
\omega =  dx \wedge \theta^0 \wedge \epsilon - dt \wedge \beta , \quad \epsilon \in  \Omega^{0,1}(\CR) ,\ \beta \in \Omega^{0,2}(\CR)
$$
where $\epsilon^*=(-X)^i(r_i \theta^0) = - \epsilon$, then $\CL^*_\Delta(\epsilon)=0$ and $\beta= \bbeta(\epsilon)$ is given by equation \ref{CF2}.
\end{Corollary}

\begin{proof}  Write $\beta = B_{ab} \theta^a \wedge \theta^b$, and 
choose in equation \ref{pbw0}, $\tilde \omega_0 =dx \wedge \vartheta^0 \wedge \tilde \epsilon_0- dt\wedge \tilde \beta$, where $\tilde \epsilon_0 =  \tr_{i} \vartheta^{i}, \tilde \beta= B_{ab} \vartheta^a \wedge \vartheta^ b $ (see Remark \ref{thm_rmk}).  
Then there exists $s_{ab} \in C^\infty(\CJ)$ such that,
\begin{equation}
\begin{aligned}
d_\sH(\tilde \omega_0) &=dt \wedge dx \wedge (  \vartheta_t \wedge \tilde \epsilon_0 + \vartheta_0 \wedge D_t(\tilde \epsilon_0) + D_x(\tilde \beta) ) \\
&= dt \wedge dx \wedge (  (d_\sV\Delta+ K_m \vartheta^m) \wedge \tilde \epsilon_0+ \vartheta_0  \wedge D_t(\tilde \epsilon_0)+D_x(\tilde \beta) )\\
&= dt \wedge dx \wedge (  (d_\sV\Delta+ K_m \vartheta^m) \wedge \tilde \epsilon_0 + \vartheta_0 \wedge (\tr_{i,t} \vartheta^i + \tr_i \vartheta^i_t)+D_x(\tilde \beta))\\
&= dt \wedge dx \wedge ( d_\sV\Delta \wedge \tilde \epsilon_0+ \vartheta^0 \wedge(\tr_i D_x^i( d_\sV \Delta)) +s_{ab} \vartheta^a \wedge \vartheta^b )
\end{aligned}
\label{TT2}
\end{equation}
since $D_t(\vartheta^0)= d_\sV\Delta+ K_m \vartheta^m$.  Now using equations \ref{XB1} and \ref{XB2} with $X=D_x$, $\rho=\vartheta^0$, $K_i=r_i$, and $\theta^0= d_\sV \alpha$ while adding and subtracting $\tr_0 \vartheta^0 \wedge d_\sV \Delta$ we have
\begin{equation}
dt \wedge dx \wedge \vartheta^0 \wedge ( \tr_i D_x^i( d_\sV \Delta))  = 
dt \wedge dx \wedge   (-D_x)^i( \tr_i \vartheta^0) \wedge  d_\sV \Delta - d_\sH\tilde \eta 
\label{TT3}
\end{equation}
where
\begin{equation}
\tilde \eta=  dt \wedge 
\sum_{i=1}^k \sum_{j=1}^i (-D_x)^{j-1}(\tr_i \vartheta^0) \wedge D_x^{i-j}( d_\sV \Delta) 
\label{TT3b}
\end{equation}
and $\tilde \eta$ satisfies $\iota^* \tilde \eta_0 =0$. Since $\epsilon ^* = -\epsilon$, we have $\tilde\epsilon_0^* = - \tilde \epsilon_0$, and combining this with equation \ref{TT2} and \ref{TT3} we have
\begin{equation}
d_\sH(\tilde \omega_0) = dt \wedge dx \wedge (  d_\sV\Delta \wedge 2 \tilde \epsilon_0+s_{ab} \vartheta^a \wedge \vartheta^b ) - d_\sH \tilde \eta
\label{TT4}
\end{equation}
Therefore comparing equations \ref{TT4} with equation \ref{r_hot} we have $\tilde \rho=2 \tilde \epsilon_0$. By equations  \ref{endom} in the proof of Theorem \ref{Thmcf1}, $\rho= \frac{1}{2} \iota^* \tilde \rho  = \epsilon$ satisfies $\CL^*_\Delta(\epsilon)=0$.  Finally Theorem \ref{Thmcf2} implies $\beta= \bbeta(\epsilon)$.
\end{proof}

\section{ Canonical forms for  $H^{1,2}(\CR)$} \label{CFU}
 
By  refining Theorem \ref{Thmcf1} we will produce a canonical form for elements of $H^{1,2}(\CR)$, by determining a unique representative. A form $\rho  \in \Omega^{0,1}(\CR)$ can always be written $\rho = r_iX^i(\theta^0)$. We define the adjoint of $\rho$ by $\rho^* = (-X)^i(r_i \theta^0)$ while $(\rho^*)^*=\rho$ because the operator $r_i X^i$ has this property, see Remark \ref{thm_rmk}.

\begin{Theorem} \label{Urep} Let $[\omega] \in H^{1,2}(\CR)$. There exists a unique representative having the form
\begin{equation}
\omega = dx \wedge \theta^0 \wedge \epsilon - dt \wedge \bbeta(\epsilon)\ , \quad \epsilon \in \Omega^{0,1}(\CR)
\label{CFO}
\end{equation}
where $ \epsilon^*  =  -\epsilon$, and $\bbeta(\epsilon) $ is given by the formula in \ref{CF2} and $\theta^0 \wedge \CL_\Delta^*(\epsilon)=0$.
\end{Theorem}
\begin{proof}  We begin by utilizing equation \ref{XB2} and make the substitution $\rho = \theta^0$, $K_i = r_i$ giving the identity,
\begin{equation}
X \left( \sum _{i=1}^k  \left( \sum_{a=1}^{i} (-X)^{a-1}(r_{i} \theta^0)\wedge \theta^{i-a} \right) \right)= 
\sum _{i=1}^k \left(-(-X)^i(r_i \theta^0) \wedge \theta^0 + r_i\theta^0 \wedge \theta^i \right).
\label{FXH1}
\end{equation}
If we now write $\rho= \sum_{i=1}^k r_i \theta^i $ and let
\begin{equation}
\eta=\sum _{i=1}^k  \left( \sum_{a=1}^{i} (-X)^{a-1}(r_{i} \theta^0)\wedge \theta^{i-a} \right),
\label{rad1}
\end{equation}
the identity \ref{FXH1} gives
\begin{equation}
X( \eta) = \theta^0 \wedge( \rho ^* +\rho ) 
\label{rad}
\end{equation}

Suppose now $[\omega ] \in H^{1,2}(\CR)$ with representative $\omega=dx \wedge \theta^0 \wedge \rho- dt \wedge \bbeta(\rho)$ with $\rho=r_i\theta^i$ from Theorem \ref{CF1}. Let $\hat \omega = \omega - \frac{1}{2} d_\sH(\eta)$ so that $[\hat \omega]=[\omega]$ and where $\eta $ is given in equation \ref{rad1}.  We  then use  equation  \ref{rad} to replace $X(\eta)$ in the following,
\begin{equation}
\begin{aligned}
\hat \omega  & = dx \wedge \theta^0 \wedge \rho - dt \wedge \bbeta(\rho) - \frac{1}{2}dx \wedge X(\eta) - \frac{1}{2} dt \wedge T(\eta) \\
& = dx \wedge  ( \theta^0 \wedge \rho -\frac{1}{2} X(\eta)) - dt \wedge ( \bbeta(\rho) + \frac{1}{2} T(\eta)) \\
& =  dx \wedge \left(\theta^0 \wedge \rho - \frac{1}{2} ( \theta^0 \wedge \rho +\theta^0 \wedge \rho^* )\right)  - dt \wedge ( \bbeta(\rho)+\frac{1}{2} T(\eta)) \\
& = dx \wedge \theta^0 \wedge \frac{1}{2}( \rho - \rho^*) - dt \wedge (\bbeta(\rho)+\frac{1}{2}T(\eta)).
\end{aligned} 
\label{oh3}
\end{equation}
The representative $\hat \omega$ in \ref{oh3}  satisfies the skew adjoint condition in the theorem with 
$$
\epsilon =\frac{1}{2}(\rho -\rho^*),
$$
while Corollary \ref{CFbeta} shows $\bbeta(\epsilon) = \bbeta(\rho)+\frac{1}{2} T(\eta)$. 

We now show the representative \ref{CFO} unique. Suppose that
\begin{equation}
\omega_{\alpha} = dx \wedge \theta^0 \wedge  \epsilon_\alpha -dt \wedge \check \bbeta(\epsilon_a), \qquad \alpha=1,2
\label{hrhocf}
\end{equation}
where $ \epsilon_\alpha^* = - \epsilon_\alpha$ and $[\omega_1]=[\omega_2]$. This implies
there exists $\xi= \xi_{ab} \theta^a \wedge \theta^b \in \Omega^{0,2}(\CR)$ such that
\begin{equation}
dx \wedge \theta^0 \wedge (\epsilon_1-\epsilon_2) = dx \wedge X(\xi_{ab} \theta^a \wedge \theta^b ).
\label{Unp1}
\end{equation}

Now let $\tilde \epsilon_\alpha= r_{i,\alpha} \vartheta^i \in \Omega^{0,1}(\CJ)$ and $\tilde \xi =\xi_{ab} \vartheta^a \wedge \vartheta^b \in \Omega^{0,2}(\CJ)$ so that $\iota^* \tilde \epsilon_\alpha = \epsilon_\alpha$, $\iota^* \tilde \xi = \xi$. Equation \ref{Unp1} implies
\begin{equation}
dt \wedge dx \wedge \vartheta^0 \wedge (\tilde \epsilon_1 -\tilde \epsilon_2) =dt \wedge [D_x (\tilde \xi) ]=  - d_\sH( dt \wedge \tilde \xi).
\label{dheq}
\end{equation}
Applying the integration by parts operator $I$ (using \ref{IEP}) to equation \ref{dheq} and that $\tilde \epsilon^* _\alpha= -\tilde \epsilon_\alpha$ gives
$$
dt \wedge dx \wedge \theta^0 \wedge (\tilde \epsilon_1 -\tilde \epsilon_2)=0.
$$
Since $\tilde \epsilon_1-\tilde \epsilon_2$ is skew-adjoint, this implies $\tilde \epsilon_1=\tilde \epsilon_2$
and that $\epsilon_1=\epsilon_2$. This implies $\bbeta(\epsilon_1)=\bbeta(\epsilon_2)$ and so $\omega_1=\omega_2$.
\end{proof}

\begin{Corollary}\label{Curep} If $[\omega] \in H^{1,2}(\CR)$ with representative $\omega=dx \wedge \theta^0 \wedge \rho-dt \wedge \bbeta(\rho)$ then the unique representative in Theorem \ref{Urep} has
\begin{equation}
\epsilon = \frac{1}{2} \left( \rho - \rho^*\right).
\end{equation}
\end{Corollary}

We now refine Theorem \ref{Thmcf1} and provide a third (non-unique) normal form.

\begin{Theorem} \label{TCF3}  Let  $[\omega] \in H^{1,2}(\CR)$ then there exists a representative $\omega$ such that 
\begin{equation}
\omega= dx \wedge  \theta^0  \wedge d_\sV \AQ   - dt \wedge d_\sV \gamma = d_\sV\left( dx \wedge \theta^0 \cdot  \AQ   - dt \wedge \gamma \right)
\label{eta1}
\end{equation}
where $\AQ$ is a smooth function  on $\CR$ and $\gamma \in \Omega^{0,1}(\CR)$.
\end{Theorem}
\begin{proof}  We start with equation \ref{CF1} in Theorem \ref{Thmcf1}
 where a representative for $[\omega]$ can be written
$$
\omega =dx \wedge \theta \wedge  \rho - dt \wedge   \beta ,
$$
where w.l.o.g. $\rho = r_a\theta^a , a=1,\ldots,m$. Now $d_\sV \omega \in H^{1,3}(\CR)$, and so by Theorem \ref{ZC} there exists $\xi\in \Omega^{0,3}(\CR)$ such that
$$
d_\sV  \omega  = d_\sH \xi.
$$
Writing $\xi = A_{ijk} \theta^i \wedge \theta^j\wedge 
\theta^k $, this gives
\begin{equation}
dx \wedge X( A_{ijk} \theta^i \wedge \theta^j\wedge \theta^k)   = dx \wedge \theta^0 \wedge (d_\sV \rho).
\label{fix}
\end{equation}
We now show $\xi$ has the form
\begin{equation}
\xi = A_i \theta^ i \wedge \theta ^1 \wedge \theta^0, \quad i=0,\ldots, m-1
\label{cdxi}
\end{equation}
Suppose there is a term in $\xi$ with $\theta^{M_1}\wedge \theta ^{M_2} \wedge \theta ^{M_3}$, $1 \leq M_1 < M_2 < M_3$, and assume we have the one with the highest $M_3$. On the left
side of \ref{fix} there will be 
$$
dx \wedge X(  \theta^{M_1}\wedge \theta ^{M_2} \wedge \theta ^{M_3})
$$
which contains $dx \wedge \theta^{M_1}\wedge \theta ^{M_2} \wedge \theta ^{M_3+1}$, which
can't occur on the right side since there is no $\theta^0$.  Suppose now that there are terms in $\xi$ of the form $\theta^0 \wedge \theta ^{M_2} \wedge \theta ^{M_3}$ 
with $1<M_2 < M_3$. Consider the maximal $M_3$, and again 
$$
dx \wedge X(  \theta^0 \wedge \theta ^{M_2} \wedge \theta ^{M_3}),
$$
will contain a term   $dx \wedge \theta^0 \wedge \theta ^{M_2} \wedge \theta ^{M_3+1}$ which can't occur on the right hand side of equation \ref{fix}. This shows equation \ref{cdxi}.

Now
$$
d_\sH d_\sV \xi = 0
$$
but since $\xi \in \Omega^{0,3}(\CR)$, this implies $d_\sV\xi = 0$. We apply vertical exactness, and let $\zeta \in \Omega^{0,2}(\CR)$ be such that
$ d_\sV \zeta = \xi $.  By the vertical homotopy on $\xi$ we may assume $\zeta= A \theta^0\wedge \theta^1$
Finally, we let 
\begin{equation}
\tilde \omega = \omega + d_\sH \zeta= dx \wedge  \theta^0 \wedge \tilde \rho - dt \wedge \tilde \beta ,
\label{to}
\end{equation}
where $\tilde \rho =  \rho +X(A\theta^1 \wedge)$ and $\tilde \beta = \beta - T(\zeta)$. Therefore, 
$$
d_\sV\tilde \omega = d_\sV \omega + d_\sV d_\sH \zeta =d_\sV \tilde \omega -d_\sH d_\sV \zeta =d_\sV \tilde \omega -d_\sH \xi= 0.
$$  
This proves there is a representative $\tilde \omega$ for $[\omega]$ with $d_\sV\tilde \omega=0$.

Again we use  $d_\sV$ exactness to find $\eta\in \Omega^{1,1}(\CR)$ such that,
\begin{equation}
\tilde \omega = d_\sV \eta 
\label{dVo}
\end{equation}
where 
\begin{equation}
\eta  =  dx\wedge \alpha  -dt \wedge \gamma ,\qquad \alpha,\gamma \in \Omega^{0,1}(\CR).
\label{dVo2}
\end{equation}
Writing $\alpha= a_j \theta^j$, $j=0,\ldots,m$, equation \ref{dVo} and \ref{dVo2} give
$$
\theta^0 \wedge \tilde \rho = - d_\sV ( a_j \theta^j).
$$
We now modify $\eta$ in equation \ref{dVo2} and the representative $\tilde \omega$ for $[\omega]$ in equation \ref{to}  by
$$
\hat \eta = \eta - d_\sH ( a_m \theta^ {m-1} ),\qquad \hat \omega = \tilde \omega + d_\sH d_\sV( a_m \theta^{m-1}), \quad {\rm no \ summation}
$$
so that $\hat \omega = d_\sV \hat \eta$. In particular we note
$$
\hat \eta = dx \wedge (\hat a_j \theta ^j) - dt \wedge \hat \gamma \qquad j= 0, \ldots m-1.
$$
Continuing by induction, there exists a representative $\bar \omega$ for $[\omega]$ and an $\bar \eta \in \Omega^{1,1}(\CR)$, where $\bar \omega = d_\sV \bar \eta$ and
\begin{equation}
\bar \eta  =  dx\wedge \theta^0 \cdot  \AQ  -dt \wedge \bar \gamma 
\label{cfe2}
\end{equation}
where $\AQ$ is a smooth function on $\CR$. This also implies
$$
\bar \omega = dx \wedge \theta^0 \wedge d_\sV Q - dt \wedge d_\sV \bar \gamma .
$$
\end{proof}

Combining Theorem \ref{TCF3} and Corollary \ref{Curep} gives the following.

\begin{Corollary}\label{Curep2} If $[\omega] \in H^{1,2}(\CR)$ with representative $\omega=d_\sV( dx \wedge \theta^0 \cdot Q- dt \wedge \gamma)$ from Theorem \ref{TCF3} then the unique representative in Theorem \ref{Urep} is determined by
\begin{equation}
\epsilon = \frac{1}{2} \left( \CL_Q - \CL_Q^*\right)\theta^0
\label{epsQ}
\end{equation}
where $ \CL_Q = Q_i X^i $. 
\end{Corollary}

The snake lemma from the variational bicomplex is the following.

\begin{Lemma} \label{exlam} Let $\omega \in \Omega^{1,2}(\CR)$ and $d_\sH\omega=0$, $d_\sV\omega=0$. Let $\eta \in \Omega^{1,1}(\CR)$ such that $d_\sV \eta= \omega$. Then
there exists $\lambda =L \, dt \wedge dx \in \Omega^{2,0}(\CR)$  such that
$d_\sH \eta = d_\sV \lambda$.
\end{Lemma}
\begin{proof}  We have
$$
d_\sV d_\sH \eta= - d_\sH d_\sV \eta= -d_\sH \omega = 0,
$$
and the vertical exactness of the bicomplex immediately implies the lemma.
\end{proof}
The relationship between $\omega$, $\lambda$ and $\eta$ in Lemma \ref{exlam} is represented by the diagram,
\begin{equation*}
\begin{gathered}
\begindc{\commdiag}[30]
\obj(0, 25)[Z]{$0$}
\obj(0, 12)[KN]{$\omega$}
\obj(0, 0)[HD]{$ \eta $\,}
\obj(15, 0)[PHD]{$d_\sH \eta $}
\mor{KN}{Z}{$ d_\sV$}[\atright, \solidarrow]
\mor{HD}{KN}{$ d_\sV$}[\atright, \solidarrow]
\mor{HD}{PHD}{$d_\sH$}[\atright, \solidarrow]
\obj(15, -12)[PL]{$\lambda$}
\mor{PL}{PHD}{$d_\sV$}[\atright, \solidarrow]
\enddc
\end{gathered}
\end{equation*}


\begin{Corollary}\label{LamHom} Let $[\omega] \in H^{1,2}(\CR)$ and let $\omega $ be the $d_\sV$ closed representative from equation \ref{eta1}, and let $\lambda \in \Omega^{2,0}(\CR)$ be as in Lemma \ref{exlam}, so that $d_\sH \eta = d_\sV \lambda$. 
The  linear map $\Lambda :H^{1,2}(\CR) \to H^{2,0}(\CR)$ given by 
\begin{equation}
\Lambda([\omega]) = [\lambda]
\label{defLambda}
\end{equation} 
is well defined.
\end{Corollary}

\begin{proof}  Suppose $\omega_a \in \Omega^{1,2}(\CR), a=1,2$ where $[\omega_1]=[\omega_2]$ and that $\omega_a=d_\sV \eta_a, a=1,2$ are two $d_\sV$ closed representatives. Let $\lambda_a\in \Omega^{2,0}(\CR)$ satisfy $d_\sH\eta_a = d_\sV \lambda_a, a=1,2$. Since $[\omega_1]=[\omega_2]$, there exists $\xi \in \Omega^{0,2}(\CR)$ such that,
$$
d_\sV(\eta_1 -\eta_2) = d_\sH \xi
$$
This implies $ d_\sV d_\sH \xi = -d_\sH d_\sV \xi = 0$ which then implies $d_\sV \xi = 0$ since $H^{0,2}(\CR) = 0$ (or $d_\sH \mu =0, \mu \in \Omega^{0,s}(\CR) $ implies $\mu=0$). Therefore there exists $\phi \in \Omega^{0,1}(\CR)$ such that
$$
\xi = d_\sV \phi
$$
which implies
$$
d_\sV (\eta_1 -\eta_2) = d_\sH( d_\sV \phi) 
$$
and that
$$
d_\sV( \eta_1- \eta_2 + d_\sH \phi) = 0.
$$
By vertical exactness there exists $\kappa\in \Omega^{1,1}(\CR)$ such that
$$
\eta_1 - \eta_2 +d_\sH \phi = d_\sV \kappa
$$
Taking $d_\sH$ of this equation gives
$$
\begin{aligned}
d_\sH \eta_1- d_\sH \eta_2  &= d_\sH d_\sV \kappa \\
d_\sV \lambda_1 - d_\sV \lambda_2 & = - d_\sV d_\sH \kappa
\end{aligned}
$$
so
$$
d_\sV( \lambda_1 - \lambda_2 + d_\sH \kappa) = 0.
$$
Therefore
\begin{equation}
\lambda_1 - \lambda_2 = d_\sH \kappa + \mu
\label{dellam}
\end{equation}
where $\mu \in \Omega^2(\reals^2)$. The deRham cohomology of $\reals^2$ is trivial so $\mu = d \alpha= d_\sH\alpha, \alpha \in \Omega^1(\reals^2)$. Therefore equation \ref{dellam} becomes
$$
\lambda_1 - \lambda_2 = d_\sH (\kappa +\alpha)
$$
and $[\lambda _1]=[\lambda_2]\in H^{2,0}(\CR)$. 
\end{proof}
The relevance of the kernel of $\Lambda$ is given in Theorem \ref{kerPi}.

\section{Variational Operators and $H^{1,2}(\CR)$}\label{VOH12}

A scalar evolution equation with  $\Delta =u_t- K(t,x,u,u_x,\ldots, u_n)$ is said to admit a variational operator  of order $k$ if there exists a differential operator
$$
{\CE}=\sum_{i=0}^k r_i(t, x, u, u_x,u_{xx}, \ldots)  D_x^i 
$$ 
$r_k \neq 0$, and a function $L \in C^\infty(\reals^2, \reals)$ such that,
\begin{equation}
{\CE}( \Delta ) = \Eop\left( \, L(t,x,u,u_t,u_x,u_{tt},u_{tx},u_{xx},\ldots)\ \right)
\label{VOPe}
\end{equation}
where $\Eop( L)$ is the Euler-Lagrange expression of $ L\in C^\infty(J^\infty(\reals^2,\reals))$. We will prove a generalization  of Theorem 2.6 in \cite{anderson-thompson:1992a} which relates the existence of a multiplier (or zero order operator) to the cohomology of $\Delta$. We start with the following theorem.

\begin{Theorem}\label{e_o}  Let ${ \CE}=r_i D_x^i$ be a variational operator for $\Delta=u_t-K(t, x, u, u_x, \ldots , u_n)$ with Lagrangian $ L$ satisfying \ref{VOPe}. Then there exists $ \eta \in \Omega^{1,1}(J^\infty(\reals^2,\reals))$ such that
\begin{equation}
d_\sV (L dt \wedge dx) =dt \wedge dx \wedge \vartheta^0 \cdot   {\CE}( \Delta)+ d_\sH  \eta.
\label{dVE2}
\end{equation}
With $\iota:\CR \to J^\infty(\reals^2,\reals)$ in equation \ref{defiota} let, 
$$
\omega  = d_\sV(\, \iota^* \eta\, ) \in \Omega^{1,2}(\CR).
$$
Then $d_\sH\omega=0$. 
\end{Theorem}

\begin{proof}  Suppose $\CE$ and $L$ are given satisfying \ref{VOPe}, then using the standard formula in the calculus of variations (for example equation (3.2) in \cite{anderson:1992a}), we have on account of \ref{VOPe}
\begin{equation}
\begin{aligned}
d_\sV (L dt \wedge dx ) & = dt \wedge dx \wedge \vartheta^0 \cdot \Eop(L)+ d_\sH \eta \\
&= dt \wedge dx \wedge \vartheta^0 \cdot \CE(\Delta)+ d_\sH \eta
\end{aligned} 
\label{dVL}
\end{equation}
which shows \ref{dVE2}.

By applying $\iota^*$ to equation \ref{dVE2} we have
\begin{equation}
 d_\sV  \iota^* \lambda = d_\sH \iota^*  \eta.
\label{dVlam}
\end{equation}
Letting $\omega = d_\sV \iota^*\eta$, we compute $d_\sH \omega$ using equation \ref{dVlam} and get
$$
d_\sH \omega = d_\sH d_\sV \iota^* \eta = -d_\sV d_\sH\iota^*  \eta= -d_\sV (\iota^* d_\sV \lambda)) = -d_\sV^2(\iota^*\lambda) = 0.
$$ 
Therefore $d_\sH \omega=0$. 
\end{proof}

A formula for  $\omega$ in terms of $ \CE$ in Theorem \ref{e_o} is given in Theorem \ref{rho_e} just below. Before giving the theorem we note the following property of variational operators for evolutions equations.

\begin{Lemma} If $\Delta = u_t -K(t,x,u,u_x,\ldots, u_n)$ admits the $k^{th}$ order variational operator $
{\CE}=  r_i(t,x,u,u_x,\ldots) D_x^i,\ i =0, \ldots,k$  then ${\CE}$ is skew adjoint.
\end{Lemma}

\begin{proof}  Suppose ${\CE}( u_t -K) = \Eop(L)$ then applying $I \circ d_\sV $ to equation \ref{dVE2} using  $d^2_\sV=0$ along with the property \ref{IEP3} for $I$, we have
\begin{equation}
I \circ d_\sV\left(\ dt \wedge dx\wedge \vartheta^0 \cdot {\CE}( \Delta )    \ \right)=0.
\label{IdVZ}
\end{equation}
With $\Delta = u_t -K$ let
\begin{equation}
\kappa =d_\sV\left(   \cdot {{\CE}}( \Delta )\right) =
  r_i D_x^iD_t(\vartheta^0)  +u_{t,i} d_\sV r_i  - d_\sV( r_i D_x^i(K)) .
\label{om2}
\end{equation}
so that condition \ref{IdVZ} gives
\begin{equation}
2 I(dt \wedge dx \wedge  \vartheta^0 \wedge \kappa) = dt \wedge dx \wedge  \vartheta^0 \wedge \left( \kappa \\
-  \sum_{(a,j)\neq(0,0)}^\infty  (-1)^{j+a} D_x^jD_t^a( \vartheta^0 \cdot \partial_{u_{a,j}} \hook \kappa) \right)
\label{IdVp}
\end{equation}
In the term $\partial_{u_{a,j}} \hook \kappa$ where $\kappa$ is given in equation \ref{om2}
we note that $\partial_{u_{a,j}}( \tr_j)=0, \partial_{u_{a,j}}(K)=0, a\geq 1,j\geq 0$. Therefore
the only possible non-zero terms in the summation term in equation \ref{IdVp}  with $\partial_{u_{a,j}}\hook $ with $ a\geq 1,j\geq 0$ satisfy
\begin{equation}
\begin{aligned}
-D_t (\partial_{u_t} \hook \kappa)&= -D_t( r_0 \vartheta^0) &\equiv&\  - r_0 \vartheta_t   \  & & \mod \{ \, \vartheta^j \, \}_{j \geq 0} \\
 D_xD_t (\partial_{u_{t,1}}  \hook \kappa)&=D_xD_t( r_1 \vartheta^0) &\equiv &\ D_x(  r_1 \theta_t)  \ & & \mod \{ \, \vartheta^j \, \}_{j\geq 0} \\
-(-1)^k D_x^{k}D_t(\partial_{u_{t,1}}  \hook \kappa)&=-(-1)^k D_x^k D_t( r_k \vartheta^0)& \equiv &\ (-1)^k D_x^k( r_k \theta_t) \  & & \mod \{ \, \vartheta^j \, \}_{j\geq 0} .
\end{aligned}
\label{exIntE}
\end{equation}
Writing the condition $I(dt \wedge dx \wedge \vartheta^0 \wedge \kappa) \mod \{ \vartheta^j\}_{j\geq 0} $ using equation \ref{IdVp} and \ref{exIntE} gives 
\begin{equation}
\begin{aligned} 
2I(dt \wedge dx \wedge \vartheta^0 \wedge \kappa) & \equiv dt \wedge dx \wedge\vartheta^0 \wedge \left(  r_i D_x^i\vartheta_t +
\sum_{i=0}^k  (-D_x)^i(r_i  \vartheta_t)\right) & & \mod \{ \vartheta^j \}_{j\geq 0}  \\
& \equiv dt \wedge dx \wedge \vartheta^0 \wedge \left(\ {\CE}(\vartheta_t) + {\CE}^*(\vartheta_t) \ \right)  & & \mod \{ \vartheta^i\}_{i\geq 0} .
\end{aligned}
\label{FEQ1}
\end{equation}
In order for the right side of equation \ref{FEQ1} to be zero we must have ${\CE}^* = - {\CE}$.
\end{proof}

\begin{Theorem}\label{rho_e} Let $\CE=r_i(t,x,u,u_x,\ldots) D_x^i, \ i=0,\ldots, k$ be a $ k^{th}$ order variational operator for $\Delta=u_t-K(t, x, u, u_x,\ldots,u_n)$ and let $[d_\sV \iota^* \eta] \in  H^{1,2}(\CR) $ from Theorem \ref{e_o}. Then the unique representative for $[d_\sV \iota^* \eta]$ in Theorem \ref{Urep} is
\begin{equation}
\omega = dx \wedge   \theta \wedge  \epsilon  - dt \wedge \bbeta(\epsilon), \quad \epsilon = - \frac{1}{2}\iota^*\CE(\vartheta^0) =-\frac{1}{2} r_i\theta^i.
\label{cf_E}
\end{equation}
\end{Theorem}

\begin{proof}  Let $\hat \omega_0 = d_\sV\tilde \eta $ where $\eta$ satisfies equation \ref{dVL}. We have
from equations \ref{dVL} and \ref{TT3}
\begin{equation}
\begin{aligned}
d_\sH \tilde \omega_0  &= 
-d_\sV d_\sH(\eta)=  dt \wedge dx \wedge d_\sV(\vartheta^0\cdot \CE(\Delta))  \\
&= 
dt \wedge dx \wedge  \left(  D_x^i(\Delta) d_\sV\tr_i \wedge \vartheta^0  +  r_i D_x^i( d_\sV \Delta) \wedge \vartheta^0 \right) \\
&=dt \wedge dx \wedge  \left(  D_x^i(\Delta) d_\sV\tr_i \wedge \vartheta^0  + d_\sV \Delta \wedge (- D_x)^i( r_i \vartheta^0)  \right)+ d_\sH \tilde \eta
\end{aligned}
\label{dV_to}
\end{equation}
where $\tilde \eta $ is given in equation \ref{TT3b} and satisfies $\iota^* \tilde \eta=0$.
As remarked in the 
proof of Theorem \ref{Thmcf1}, the term $ D_x^i(\Delta) d_\sV\tr_i \wedge \vartheta^0$ in \ref{dV_to} 
does not contribute to the form $\tilde \rho$ in equation \ref{r_hot}. Therefore we have
from equation \ref{dV_to}  that $\tilde \rho$ in equation \ref{r_hot} is,
$$
\tilde \rho = (- D_x)^i(\tr_i \vartheta) =  \CE^*( \vartheta^0).
$$
Since $\rho = \frac{1}{2} \iota^* \tilde \rho $ and $ \CE$ is skew-adjoint we get equation \ref{cf_E}. 
\end{proof}

We now come to the last main theorem in this section which proves the converse to Theorem \ref{e_o}. 
The proof is again a generalization of the argument given in  Theorem 2.6 of \cite{anderson-thompson:1992a} for the multiplier problem. 

\begin{Theorem}\label{conf_2}  Let $[\omega] \in H^{1,2}(\CR)$ with representative $\omega$ as in equation \ref{eta1} in Theorem \ref{TCF3} given by
\begin{equation}
\omega =dx\wedge \theta \wedge \rho-dt \wedge \bbeta(\rho) = d_\sV\eta 
\label{conf_2e}
\end{equation}
where 
\begin{equation}
\eta = dx \wedge \theta \cdot \AQ - dt \wedge \gamma.
\label{thm_eta}
\end{equation}
Let $\lambda= L dt \wedge dx$  satisfying $d_\sH\eta=d_\sV \lambda$ from Lemma \ref{exlam}.
Then $\CE= {\bf F}_{\tQ}^*-{\bf F}_{ \tQ}$ is a variational operator and,
\begin{equation}
 \CE(\Delta) = ({\bf F}_{\tQ}^*-{\bf F}_{ \tQ})(\Delta) = \Eop \left(\tQ \Delta +\tL \right).
\label{thm_m}
\end{equation}
\end{Theorem}

The proof requires considerable care whether we are working on $\CR$ or $J^\infty(\reals^2,\reals)$, see Remark \ref{thm_rmk}.

\begin{proof} We start by writing $\gamma= g_{j} \theta^j$ in equation \ref{thm_eta} and define the form 
$\tilde \eta \in J^\infty(\reals^2,\reals))$ given by
$$
\tilde \eta = dx \wedge \vartheta^0 \cdot \tQ  - dt \wedge [g_{j} \vartheta^j ]
$$
where $\iota^* \tilde \eta = \eta$ in equation \ref{thm_eta},  and the forms $\vartheta^j$ are defined in \ref{VTD}. Now define the vector fields on $J^\infty(\reals^2,\reals)$,
\begin{equation}
\begin{aligned}
\tilde T &=  \partial_t + K \partial_u +  D_x(K)  \partial_{u_x} + D_x^2(K) \partial_{u_{xx}} + \ldots,\!\!
 &&V  = \Delta \partial_u + D_x( \Delta) \partial_{u_{x}}+ \ldots, \\
\tilde X &= \partial_x + u_x \partial_u  + u_{xx}\partial_{u_x} +D_x(K) \partial_{u_t} + \ldots, \!\!
&&W = D_x(\Delta)\partial_{u_t}  + D_x^2(\Delta) \partial_{u_{tx}}+\ldots 
\end{aligned}
\label{defTXVW}
\end{equation}
so that $D_t= \tilde T+  V, D_x=\tilde X+ W$. Then
\begin{equation}
\begin{aligned}
d_\sH \tilde \eta  & = dt \wedge dx \wedge [ D_t ( \vartheta^0\cdot \tQ )  + D_x(\tg _{j} \vartheta^j ) ] \\
&= dt \wedge dx \wedge[ \tilde T( \vartheta^0 \cdot \tQ) + \tilde X( \tg _{j} \vartheta^j)] + dt \wedge dx \wedge [V(\vartheta^0 \cdot \tQ) +W(g _{j} \vartheta^j)] \\
\end{aligned}
\label{dHvp0}
\end{equation}
Since $g_j = g_j(t, x, u, u_x,\ldots) $ then $W( g_j)=0$, while $dt \wedge dx \wedge D_x( \vartheta^j)=dt \wedge dx \wedge \tilde X(\vartheta^j)$ and $V(\vartheta^0) = d_\sV \Delta$. Equation \ref{dHvp0} then can be written,
\begin{equation}
d_\sH \tilde \eta = dt \wedge dx \wedge[ \tilde T(\tQ \vartheta^0) +\tilde X( \tg _{j} \vartheta^j)] + dt \wedge dx \wedge [ \vartheta^0\cdot V(\tQ )  + d_\sV \Delta \cdot \tQ ] . \\
\label{dHvp1}
\end{equation}

The condition $d_\sH \eta = d_\sV \lambda$ (on $\CR$) is
\begin{equation}
dt \wedge dx \wedge [ T (  \theta\cdot \AQ) - X(g_{j} \theta^j)  ]=dt \wedge dx \wedge \theta^a \cdot L_a
\label{dHid}
\end{equation}
on $\CR$. Now $\pi^* dt \wedge dx \wedge \theta^j = dt \wedge dx \wedge \vartheta^j$ (see Remark \ref{thm_rmk}), and using
the vector fields in \ref{defTXVW} we have
\begin{equation}
\begin{aligned}
\pi^* (d_\sV \lambda) &= d_\sV (\pi^* \lambda)  = d_\sV (\tL dt \wedge dx),\\
\pi^*( dt \wedge dx \wedge [ T (\AQ  \theta^0)]&= dt \wedge dx \wedge [ \tilde T (\tQ  \vartheta^0)]\\
\pi^*( dt \wedge dx \wedge [ X(g_{j} \theta^j)]&= dt \wedge dx \wedge [ \tilde X(g_{j} \vartheta^j)]\\
\end{aligned}
\label{pb2}
\end{equation}
Therefore applying $\pi^*$ to  \ref{dHid} and using \ref{pb2} we have
\begin{equation}
 dt \wedge dx \wedge[ \tilde T(\tQ \vartheta^0) + \tilde X(g_{j} \vartheta^j )] =d_\sV (\tL dt \wedge dx),
\label{dvlt}
\end{equation}

The first variational formula for $d_\sV (\tL dt \wedge dx)$ on $J^\infty(\reals^2,\reals)$ applied to the right side of \ref{dvlt} gives
\begin{equation}
dt \wedge dx \wedge [\tilde T(\tQ \vartheta^0) + \tilde X( g_{j} \vartheta^j )]
= dt \wedge dx \wedge   \vartheta^0 \cdot \Eop(\tL) + d_\sH \tilde \zeta_1
\label{bp3}
\end{equation}
where $\tilde \zeta \in \Omega^{1,1}(J^\infty(\reals^2,\reals))$. Inserting equation \ref{bp3}  into \ref{dHvp1} we have, 
\begin{equation}
d_\sH \tilde \eta  =    dt \wedge dx \wedge \vartheta^0 \cdot \Eop(\tL) + dt \wedge dx \wedge [\vartheta^0 \cdot   V( \tQ)+ d_\sV \Delta \cdot \tQ ] + d_\sH \tilde \zeta_1
\label{dHvp2}
\end{equation}

The terms $ d_\sV \Delta \cdot \tQ$ in equation \ref{dHvp2} can be written as
\begin{equation}
dt \wedge dx \wedge[\tQ d_\sV \Delta]  = dt \wedge dx \wedge [d_\sV ( \tQ \Delta) - \Delta d_\sV \tQ ]
\label{rw_dV}
\end{equation}
We now apply the integration by parts operator (see equation 2.8 in \cite{anderson:1992a}) 
and use the first variational formula for $d_\sV(\tQ \Delta dt \wedge dx)$, in equation \ref{rw_dV} and 
get
\begin{equation}
\begin{aligned}
dt \wedge dx \wedge[\tQ d_\sV \Delta]  & = 
dt \wedge dx \wedge\vartheta^0 [\Eop ( \tQ \Delta) - (-D_x)^i(\tQ_{\!,i}\Delta)] + d_\sH \tilde \zeta_2\\
& = 
dt \wedge dx \wedge\vartheta^0 [\Eop ( \tQ \Delta) - {\bf F}_{\tQ}^*(\Delta)] + d_\sH \tilde \zeta_2\\
\end{aligned}
\label{rw_dV2}
\end{equation}

Next we expand the term $V(\tQ) $ in equation \ref{dHvp2} using $V$ in \ref{defTXVW} and 
\begin{equation}
V(\tQ)= D_x^i(\Delta) \tQ_{\!,i}= {\bf F}_{\tQ}(\Delta)
\label{l_t3}
\end{equation}

Inserting \ref{rw_dV2}, and \ref{l_t3} into  \ref{dHvp2} and letting $\tilde \zeta=\tilde \zeta_1+\tilde \zeta_2$ gives,
$$
d_\sH \tilde \eta = dt \wedge dx \wedge\vartheta^0\dot  [\Eop(\tL) +{\bf F}_{\tQ}(\Delta) -{\bf F}_{\tQ}^*(\Delta) +\Eop(\tQ \Delta)] +d_\sH \tilde \zeta.
$$
This implies $d_\sH(\tilde \eta -\tilde \zeta) $ is a source-form. This is only possible if   $d_\sH(\tilde \eta -\tilde \zeta) =0$, and so
$$
({\bf F}_{\tQ}^* -{\bf F}_{\tQ})(\Delta) = \Eop(\tQ \Delta +\tL),
$$
which is equation \ref{thm_m} as required.
\end{proof}

\begin{Remark} \label{findL} In general three applications of the vertical homotopy operator are required to determine $\lambda\in \Omega^{2,0}(\CR)$ from 
$[\omega] \in H^{1,2}(\CR)$. The first is to find a representative $\omega \in H^{1,2}(\CR)$ with $d_\sV \omega =0$ (Theorem \ref{TCF3}). The second is to find $\eta$ such that $d_\sV\eta=\omega$, and the third is to find $\lambda$ such that $d_\sV\lambda = d_\sH \eta$. 
\end{Remark}

We now have the following corollaries.

\begin{Corollary} \label{CU3} Let $[\omega ] \in H^{1,2}(\CR)$ with unique representative  $\omega = dx \wedge \theta^0 \wedge \epsilon - dt \wedge \bbeta(\epsilon )$, $\epsilon = r_i \theta^i$ as in Theorem \ref{Urep}. Then  $\Delta$ admits $ \CE= -2\tr_i D_x^i  $ as a variational operator. 
\end{Corollary}
\begin{proof}  Starting with equation \ref{conf_2e}, Corollary \ref{Curep2} implies
\begin{equation}
\epsilon = \frac{1}{2} (\CL_Q- \CL_Q^*) \theta^0= r_i \theta^i.
\label{WF1}
\end{equation} 
Equation \ref{WF1} together with the fact $\bFQ=\tQ_i D_x^i$ gives $\epsilon =\iota^* \frac{1}{2} (\bFQ-\bFQ^*)(\vartheta^0) = r^i \theta^i$, we have $\CE= \bFQ^*-\bFQ= -2\tr_iD_x^i$ is a variational operator by Theorem \ref{conf_2}.
\end{proof}

\begin{Corollary} Let $[\omega ] \in H^{1,2}(\CR)$ with  $\omega = dx \wedge \theta^0 \wedge (r_i\theta^i) - dt \wedge \bbeta(\rho)$ as in Theorem \ref{Thmcf1}. Then  $\Delta$ admits 
\begin{equation}
\CE=  (\tr_i D_x^i)^*-\tr_iD_x^i
\label{CEC4}
\end{equation}
as a variational operator.
\end{Corollary}
\begin{proof}   By  Corollary  \ref{Curep} the unique representative 
$ \hat \omega = dx \wedge \theta^0 \wedge \epsilon - dt \wedge \bbeta(\epsilon)$ has $\epsilon = \frac{1}{2}\left(\rho-\rho^*\right) =\frac{1}{2} ( r_i\theta^i- (-X^i)(r_i \theta^0))$. Therefore by Corollary \ref{CU3}, $\CE$ in equation \ref{CEC4} is a variational operator.
\end{proof}

Finally we may also restate Theorem \ref{conf_2} without reference to the equation manifold $\CR$ as follows.

\begin{Corollary} \label{CFQ} The operator $\CE= r_i(t,x,u,u_x,\ldots) D_x^i$, $i=0,\ldots,k$ is a variational operator for $u_t=K$ if and only if there exists $Q(t, x, u,u_x,u_{xx}, \ldots)$ and $L(t, x, u, u_x,u_{xx}, \ldots)$ such that
\begin{equation}
\CE= { \bFQ^* - \bFQ} \quad {\rm and}\quad  \CE( u_t -K) = \Eop\left( Q(u_t-K) + L \right).
\label{CEEQS}
\end{equation}
\end{Corollary}

Lastly we combine the results of Theorems \ref{e_o} and \ref{conf_2} to prove Theorem \ref{THA}.

\begin{proof} (Theorem \ref{THA})   Define the linear transformation $\hat \Phi: H^{1,2}(\CR) \to \CV_{op}(\Delta) $ 
using Theorem \ref{Urep} by
\begin{equation}
\hat \Phi([\omega])=
\hat \Phi\left([ dx \wedge \theta^0 \wedge \epsilon - dt \wedge \bbeta(\epsilon) ] \right) = - 2\tr_iD_x^i
\label{PhiI}
\end{equation}
where $\epsilon = r_i \theta^i$ and is skew-adjoint. By Corollary \ref{CU3} $\Phi([\omega])\in \CV_{op}(\Delta)$. 

We check $\hat \Phi = \Phi^{-1}$. With $\CE =-2 \tr_i D_x^i$ a variational operator, let $\epsilon = r_i \theta^i$ we have from equation \ref{FISO} (or  Theorem \ref{rho_e}) and equation \ref{PhiI},
$$
\hat \Phi \circ \Phi( \CE) = \hat \Phi\left( [dx \wedge \theta^0 \wedge \epsilon - dt \wedge \bbeta(\epsilon)] \right)= \CE
$$
and
$$
\Phi \circ \hat \Phi \left(  [dx \wedge \theta^0 \wedge \epsilon - dt \wedge \bbeta(\epsilon)] \right) = \Phi (-2\tr_i D_x^i) =[ dx \wedge \theta^0 \wedge \epsilon - dt \wedge \bbeta(\epsilon)].
$$
Therefore $\Phi$ in equation \ref{FISO} is invertible with $\hat \Phi$ in equation \ref{PhiI} as the inverse.
\end{proof}

\section{Functional $2$-Forms, Symplectic Forms and Hamiltonian Vector Fields}\label{FF2}

In this section we quickly review the space of functional forms on $ J^\infty(\reals,\reals)$ as in \cite{anderson:2016a}, \cite{anderson:1992a} and relate these to symplectic forms and symplectic operator. 

\subsection{Functional Forms}\label{FFS}

On the space $\EJ$ with coordinates $(x,u,u_x, \ldots, u_i, \ldots)$ 
the contact forms are $ \theta^i = du_i -u_{i+1} dx $ and $D_x = \partial_x+ u_x \partial_u + \ldots u_{i+1} \partial_{u_{i}}+ \ldots $ is the total $x$ derivative operator. Again  $\Omega^{r,s}(\EJ)$ denotes the $r$ horizontal, $s$ vertical forms on $\EJ$.  The horizontal and vertical differentials  $\tdH : \Omega^{r,s}(\EJ) \to \Omega^{r+1,s}(\EJ)$, $\tdV : \Omega^{r,s}(\EJ) \to \Omega^{r,s+1}(\EJ)$, satisfy 
$$
\tdH  f  = D_x( f ) dx,\quad  \tdV f =\frac{\partial f}{\partial u_{i}} \theta^i=f_i \theta^i,\quad \tdH  \theta^i  =dx \wedge  \theta^{i+1} ,\quad \tdV  \theta^i =0,
$$
where $f\in C^\infty(\EJ)$. Since $d= d_\sH+d_\sV$ this implies,
$$
 d_\sH^2=0,\quad d_\sV^2=0,\quad {\rm and}\quad \tdH \tdV +\tdV\tdH = 0.
$$

The integration by parts operator $I: \Omega^{1,s}(\EJ) \to \Omega^{1,s}(\EJ)$ is
\begin{equation}
I(\Sigma)  = \frac{1}{s} \theta^0 \wedge \sum_{i=0}^\infty  (-1)^i (D_x)^i (\partial_{u_i} \hook \Sigma), \quad \Sigma \in \Omega^{1,s}(\EJ)
\label{IE}
\end{equation}
and $I$ satisfies the same properties as in \ref{IEP2},
\begin{equation}
\Sigma  = I(\Sigma) +\tdH \eta ,\qquad I^2=I, \qquad \ker I = {\rm Image} \ d_\sH.
\label{IEP}
\end{equation}

The space of functional $s$ forms $(s\geq 1)$ on $\EJ$, $\CF^s(\EJ)\subset \Omega^{1,s}(\EJ)$, is defined to be the image of $\Omega^{1,s}(E)$ under $I$,
\begin{equation}
\CF^s(\EJ) =  I ( \Omega^{1,s}(\EJ)).
\label{Fs}
\end{equation}
By definition \ref{Fs}, equation \ref{IE} shows that any $\Sigma \in \CF^s(E)$ can always be written,
\begin{equation}
\Sigma = dx \wedge \theta^0 \wedge \alpha , \qquad \alpha \in \Omega^{0,s}(\EJ).
\label{Crep}
\end{equation}
However, not every differential form $\Sigma\in \Omega^{1,s}(\EJ)$ written as \ref{Crep} is in the space $\CF^s(\EJ)$. In the case of $\CF^2(\EJ)$ the following is easy to show using the definition of $I$ in \ref{IE}, see also Proposition 3.7 in \cite{anderson:2016a}.

\begin{Lemma} \label{LF2} Let $\Sigma \in \CF^2(\EJ)$, then there exists a unique  skew-adjoint differential operator, $ \CS = s_i D_x^i$ such that,
\begin{equation}
\Sigma = dx \wedge \theta^0 \wedge \CS(\theta^0).
\label{SKF2}
\end{equation}
\end{Lemma}

The differential $\delta_\sV:\CF^s(\EJ) \to \CF^{s+1}(\EJ) $ is defined by
$$
\delta_\sV = I \circ \tdV : \CF^s(\EJ) \to \CF^{s+1}(\EJ),\qquad s=0,\ldots
$$
where we let $\CF^0(\EJ)=\Omega^{1,0}(\EJ)$. This leads to the differential complex
\begin{equation}
C^\infty(\EJ)\ \xrightarrow{d_\sH} \ \CF^0(\EJ)\ \xrightarrow{\delta_\sV} \ \CF^1(\EJ) \ \xrightarrow{\delta_\sV} \ \CF^2(\EJ) \ \ldots ,
\label{EComp}
\end{equation}
which is exact and is known as the Euler complex, see Theorem 2.7 \cite{anderson:1992a}.

\subsection{Symplectic Forms, Symplectic Operators, and Hamiltonian Vector Fields}\label{SFS}

Let $\Gamma$ be the Lie algebra of prolonged evolutionary vector fields on $\EJ$. We begin
by recalling the appropriate definitions (see Section 2.5 \cite{dorfman:1993a}).  

\begin{Definition} \label{SD1} An element $\Sigma \in \CF^2(\EJ)$ is a {\bf symplectic form} on $\Gamma$  if $\Sigma \neq 0$ and $\delta_\sV (\Sigma)= 0$.  A skew adjoint differential operator $\CS=s_i D_x^i$ is symplectic if $dx \wedge \theta^0 \wedge \CS(\theta^0)$ is a symplectic form.
\end{Definition}

Definition \ref{SD1} combined with Lemma \ref{LF2} shows there is a one-to-one correspondence between symplectic forms and symplectic operators. We now defines Hamiltonian vector fields.

\begin{Definition} \label{HVF0} Let $\Sigma$ be a symplectic form. A vector field $Y \in \Gamma$ is Hamiltonian if
\begin{equation}
\delta_\sV \circ I(Y\hook \Sigma) =0 .
\label{LC0}
\end{equation}
\end{Definition}

Definition \ref{HVF0} is equivalent to $\Sigma$ being invariant under the flow of $Y$ on $\CF^2(\EJ)$ as shown in the following theorem.
\begin{Theorem} \label{HVF} Let $\Sigma$ be a symplectic form. An evolutionary vector field $Y \in \Gamma$ is Hamiltonian with respect to $\Sigma$ if and only if
\begin{equation}
\LD^\natural_Y \Sigma = I \circ \pi^{1,2} \circ \LD_Y \Sigma = 0,
\label{LC1}
\end{equation}
where $\LD^\natural= I \circ \pi^{1,2}\circ  \LD $ is the projected Lie derivative on $\CF^2(\EJ)$, see Theorem 3.21 in \cite{anderson:2016a}.
\end{Theorem}

\begin{proof} Using Lemma 3.24 in \cite{anderson:2016a} and the fact that $\delta_\sV \Sigma=0$, we have
\begin{equation}
\LD^\natural_Y \Sigma = I \circ d_\sV (Y \hook \Sigma) + I( Y \hook \delta_\sV \Sigma) = I \circ d_\sV (Y \hook \Sigma).
\label{CLS}
\end{equation}
By the first property in equation \ref{IEP},  $ I \circ d_\sV \circ I = I \circ d_\sV$, so conditions \ref{LC0} and \ref{LC1} are  equivalent through equation \ref{CLS}.
\end{proof}

We now write out definition \ref{HVF0} in a more familiar form. The exactness of the Euler complex and the condition $\delta_\sV\circ I ( Y\hook \Sigma)=0$ implies there exists $\lambda=2H dx \in \CF^0(\EJ)$ such that
\begin{equation}
I(Y \hook \Sigma)=\delta_\sV\lambda= dx \wedge \theta^0 \cdot \Eop(2H).
\label{Lcond}
\end{equation}
Writing $Y = \pr (K \partial_u)$ and  $\Sigma= dx \wedge \theta^0 \wedge  \CS( \theta^0)$  where $\CS= s_i D_x^i$ is a skew-adjoint differential operator. The left side of equation \ref{Lcond} is then
\begin{equation}
\begin{aligned}
I(Y \hook \Sigma) &= I(dx \wedge \left(  s_i D_x^i( K ) \theta^0 - K s_i\theta^i )  \right)   \\
&= dx \wedge \theta^0 \left(  s_i D_x^i( K ) - (-D_x)^i( Ks_i)   \right)\\
&= dx \wedge \theta^0 \cdot 2 \CS (K).  
\end{aligned}
\end{equation}
Using this computation in \ref{Lcond} shows that condition \ref{LC0} (or \ref{LC1}) is then equivalent to the 
following.

\begin{Corollary}  \label{HVFnotLP} Let $\Sigma$ be a symplectic form with corresponding symplectic operator $\CS$. The evolutionary vector field
$Y= \pr  (K\partial_u)\in \Gamma$ is Hamiltonian if and only if there exists $H\in C^\infty(\EJ)$ such that
\begin{equation}
\CS( K) = \Eop(H).
\label{HVFnotP}
\end{equation}
\end{Corollary}
Corollary \ref{HVFnotLP} just shows that Definition \ref{HVF0} agrees with the standard symplectic Hamiltonian formulation for time independent evolution equations \cite{dorfman:1993a}.

\subsubsection{Symplectic Potential}

If $\Sigma$ is a symplectic form, the exactness of the $\delta_\sV$ complex implies there exists $\psi \in \CF^1(\EJ)$ such that $\Sigma =\delta_\sV(\psi)$.  The functional form $\psi$ is a {\bf symplectic potential} for $\Sigma$.

\begin{Lemma} \label{SymPot} Let $\Sigma\in \CF^2(E)$ be symplectic (and so $\delta_\sV$ closed), then there exists a smooth function $P\in C^\infty(\EJ)$  such that
\begin{equation}
\Sigma = dx \wedge \theta^0 \wedge \CS(\theta^0),\quad {\rm where} \quad \CS =\frac{1}{2} ({\bf F}_P - {\bf F}_P^*)
\label{CF31}
\end{equation}
where ${\bf F}_P = P_i D_x^i $ is the Fr\'echet derivative of $P$.
\end{Lemma}
\begin{proof} A symplectic potential  $\psi \in \CF^1(\EJ)$ for $\Sigma$ can be written using \ref{Crep} as
\begin{equation}
\psi = dx \wedge \theta^0 \cdot P , \quad P\in C^\infty(\EJ).
\label{cfeta}
\end{equation}
Writing $\Sigma =\delta_\sV \psi $ and using equation \ref{cfeta} produces \ref{CF31}.
\end{proof}

The Hamiltonian condition on $Y\in \Gamma$ in  terms of a symplectic potential $\psi$ is the following.

\begin{Lemma} \label{LDpsi} The evolutionary vector field $Y\in \Gamma$ is Hamiltonian for the symplectic form $\Sigma=\delta_\sV \psi $ if and only if there exists $\lambda \in \CF^0(\EJ)$ such that
\begin{equation}
\LD^\natural_Y \psi = \delta_\sV \lambda.
\label{CLpsi}
\end{equation}
\end{Lemma}
\begin{proof} Using the exactness of the $\delta_\sV$ complex we show $\delta_\sV \LD ^\natural_Y \psi=0$ which is equivalent to equation \ref{CLpsi}. By Theorem \ref{HVF}, $Y$ is Hamiltonian if and only if 
\begin{equation}
0=\LD^\natural_Y \delta_\sV \psi = \delta_\sV \LD^\natural _Y \psi
\label{CLD}
\end{equation}
where we have used $\LD^\natural _Y \circ \delta_\sV = \delta_\sV \circ \LD^\natural _Y$ (Lemma 3.24 \cite{anderson:2016a}). This proves the Lemma.
\end{proof}

Using either Lemma \ref{LDpsi} or equations \ref{CF31} and  \ref{HVF} we have the following simple corollary.

\begin{Corollary} \label{HVF2C} Let $\Sigma$ be a symplectic form with symplectic potential $\psi= dx \wedge \theta^0 \cdot P$. The evolutionary vector field
$V= \pr  (K\partial_u)\in \Gamma$ is Hamiltonian if and only if there exists $H\in C^\infty(\EJ)$ such that
\begin{equation}
\frac{1}{2}({\bf F}_P-{\bf F}_P^*)(K) =  \CS(K)  = \Eop(H)
\label{HVF2}
\end{equation}
where ${\bf F}_P$ is the Fr\'echet-derivative  of $P$ on $J^\infty(\reals,\reals)$. 
\end{Corollary}

A straight forward computation writing $\Sigma = \delta_V \psi$ classifies the first order symplectic operators, see also  Theorem 6.2 in \cite{dorfman:1993a}

\begin{Lemma} An element $\Sigma\in \CF^{2}(\EJ)$ of the first order form,
$$
\Sigma= dx \wedge \theta^0 \wedge \theta^1 \cdot A  \quad  A \in C^\infty(\EJ)
$$
is symplectic, if and only if there exists $P(x,u,u_x,u_{xx}) \in C^\infty(\EJ)$
depending on up to second order derivative, such that
\begin{equation}
A=\frac {\partial P}{\partial {u_x}}- D_x \left(\frac {\partial P}{\partial u_{xx} } \right).
\label{A1sym}
\end{equation}
\end{Lemma}


\subsection{Time Dependent Systems}

Most of the definitions and results from Sections \ref{FFS} and \ref{SFS} extend immediately to the case of time dependent systems. Let $E= \reals\times J^\infty(\reals,\reals)$,  and label the extra $\reals$ with the parameter $t$. The contact forms are 
\begin{equation}
\theta^i_\sEL=du_i-u_{i+1}dx
\label{CFE}
\end{equation}
and we let $\Omega^{r,s}_{\rm t_{sb}}(E) $ be the bicomplex of $t$ semi-basic forms on $E$,
$$
\Omega^{r,s}_{\rm t_{sb}}(E) = \{\ \omega \in \Omega^{r,s}(E) \ | \ \partial_t \hook \omega=0 \ \quad ,  r=0,1; s=0\ldots \ \} .
$$
A generic form $\omega \in \Omega^{1,2}_{\rm t_{sb}}(E) $  is given by
$$
\omega= dx \wedge \theta^i_\sEL\wedge \theta^j_\sEL \cdot  \xi_{ij}, \qquad \xi_{ij} \in C^\infty(E).
$$ 
The anti-derivations $d_\sH^\sE: \Omega^{r,s}_{\rm t_{sb}}(E) \to \Omega^{r+1,s}_{\rm t_{sb}}(E) $ and $d^\sE_\sV:\Omega^{r,s}_{\rm t_{sb}}(E) \to \Omega^{r,s+1}_{\rm t_{sb}}(E) $ are determined by
\begin{equation}
d_\sH^\sE (f) = D_x(f) dx,\quad d_\sH^\sE \theta^i_E = dx \wedge \theta^{i+1}_\sE, \quad  d_\sV^\sE (f) = f_i \theta^i_\sEL, \quad d_\sV^\sE \theta^i_\sE=0,
\label{dHdVE}
\end{equation}
and satisfy $(d_\sH^\sE )^2=0, (d_\sV^\sE )^2=0, d_\sH^\sE d_\sV^\sE+d_\sV^\sE d_\sH^\sE =0$. However $d\neq d_\sH^\sE +d_\sV^\sE $. The integration by part operator $I$ induces a map $I_\sEL:\Omega^{r,s}_{\rm t_{sb}}(E)  \to \Omega^{r,s}_{\rm t_{sb}}(E) $ having the formula \ref{IE} and properties \ref{IEP}. We let 
\begin{equation}
\CF_{\rm t_{sb}}^s(E)= I_\sE \left(\Omega^{1,s}_{\rm t_{sb}}(E)\right).
\label{CFsb}
\end{equation}
The mapping $\delta_\sV^\sE= I_\sE \circ d_\sV ^\sE$ gives rise to the exact sequence as in \ref{EComp}. A form 
$\Sigma\in \CF_{\rm t_{sb}}^2(E)$ is symplectic if $\delta_\sV^\sEL \Sigma =0$ and Lemma \ref{SymPot}  becomes the following.

\begin{Lemma}\label{SyminP} An element $\Sigma=dx \wedge \theta_\sEL^0 \wedge \CS(\theta^0_\sEL) \in \CF_{\rm t_{sb}}^2(E)$ where $\CS=s_i D_x^i$, and $s_i\in C^\infty(E)$ is symplectic if and only if there exists $P \in C^\infty(E)$ such that
\begin{equation}
\CS =\frac{1}{2}( {\bf L}_P -{\bf L}_P^*)
\label{SinP}
\end{equation}
where ${\bf L}_P= P_i D_x^i$.
\end{Lemma}

We use Theorem \ref{HVF} to define Hamiltonian vector fields in this case.

\begin{Definition} \label{tHVF}
An evolutionary vector field $Y = pr(K \partial_u)$ where $K\in C^\infty(E)$ is Hamiltonian with respect to the symplectic form $\Sigma \in \CF^2_{\rm t_{sb}}(E)$ if
\begin{equation}
{\LD}^\natural_T \Sigma =I_\sE \circ \pi^{1,2} \circ {\LD}_T \Sigma =0
\label{LDT}
\end{equation}
where $ T=\partial_t +Y$ and ${\LD}^\natural_T = I_\sE \circ \pi^{1,2} \circ {\LD}_T $ is the projected Lie derivative.
\end{Definition}

Note that $T$ in Definition \ref{tHVF}  agrees with $T$ in equation \ref{DTX}. We can also write condition \ref{LDT} as follows.

\begin{Lemma} \label{IT1}  An evolutionary vector field $Y= pr (K \partial_u)$ is a Hamiltonian vector field
for the symplectic form $\Sigma=dx \wedge \theta^0_\sEL\wedge (s_i \theta^i_\sEL)$ if and only if
there exists $\xi \in \Omega^{0,2}(E)$ such that
\begin{equation}
\pi^{1,2} \circ \LD_T ( dx \wedge \theta^0_\sEL\wedge (s_i \theta^i_\sEL)) 
= dx \wedge D_x ( \xi ) = d_\sH^\sE \xi.
\label{dxi}
\end{equation}
\end{Lemma}
\begin{proof}  We have  ${\rm kernel} \ I_\sE= {\rm Image}\  d_\sH^\sE$, therefore equation \ref{LDT} can be written as equation \ref{dxi}.
\end{proof}
A formula for $\xi$ in equation \ref{dxi} in terms of $\rho=s_i\theta^i_\sEL$ is given in the proof of Theorem \ref{keyT}.

The analogue to Lemma \ref{LDpsi} also holds in this case where $Y$ is replaced by $T$. In order to prove this we now show the commutation formula in equation \ref{CLD} holds where $Y$ is replaced by $T$.

\begin{Lemma} \label{COMO}  $\LD^\natural_{T} \delta _\sV \psi = \delta_\sV \LD^\natural_{T}  \psi $. 
\end{Lemma}
\begin{proof}  Since $T=\partial_t + Y$ and $\LD^\natural_{Y} \delta _\sV \psi = \delta_\sV \LD^\natural_{Y}  \psi $ (Lemma 3.24 \cite{anderson:2016a}), we need to check
$$
\LD^\natural_{\partial_t} \delta _\sV \psi = \delta_\sV \LD^\natural_{\partial_t}  \psi 
$$
We write out both side of this equation. The left side is
\begin{equation}
I^\sE \circ \pi^{1,2} \left( \frac{1}{2} dx \wedge \theta^0_\sEL \wedge[ P_{i,t} \theta_\sEL^i -(-D_x)^i(P_{i,t} \theta^0_\sEL)] \right) =\frac{1}{2} dx \wedge \theta^0_\sEL \wedge [P_{i,t} \theta_\sEL^i -(-D_x)^i(P_{i,t} \theta_\sEL^0)].
\label{LS1}
\end{equation}
The right side is
\begin{equation}
\delta_\sV ( dx \wedge \theta^0_\sEL \cdot P_t )= \frac{1}{2} dx \wedge \theta^0_\sEL \wedge \left[  P_{t,i} \theta^i_\sEL- (-D_x)^i( P_{t,i} \theta^0_\sEL) \right].
\label{RS1}
\end{equation}
Since the mixed partials are equal $P_{t,i} = P_{i,t}$, equations \ref{LS1} and \ref{RS1} are equal, which proves the Lemma.
\end{proof}

\begin{Lemma} \label{LDpsi2} The evolutionary vector field $Y\in \Gamma$ is Hamiltonian for the symplectic form $\Sigma=\delta_\sV^\sE \psi$ if and only if there exists $\lambda \in \CF^0(E)$ such that
\begin{equation}
\LD^\natural_T \psi = \delta_\sV \lambda.
\label{CLpsi2}
\end{equation}
\end{Lemma}
\begin{proof} Using the exactness of the $\delta^\sE_\sV$ complex we show $\delta^\sE_\sV \LD^\natural_T \psi=0$ which is equivalent to equation \ref{CLpsi2}. By Definition \ref{tHVF}, $Y$ is Hamiltonian if and only if 
$$ 
0=\LD^\natural_T \Sigma=  \LD^\natural_T \delta_\sV \psi  = \delta_\sV \LD^\natural_T \psi
$$
where we have used Lemma \ref{COMO}. This proves the Lemma.
\end{proof}

Using Lemma \ref{LDpsi2}  we have the following corollary which is the $t$-dependent version of Corollary \ref{HVF2C}.

\begin{Corollary}  Let $\Sigma= dx \wedge\theta^0_\sE\wedge \CS(\theta^0_\sE)$ be a symplectic form with symplectic potential $\psi= dx \wedge \theta^0 \cdot P$. The evolutionary vector field
$Y= \pr  (K\partial_u) \in \Gamma$ is Hamiltonian if and only if there exists $H\in C^\infty(\EJ)$ such that
\begin{equation}
\frac{1}{2}\left( P_t + (\bLP-\bLP^*)(K)\right) =\frac{1}{2} P_t + \CS( K) = \Eop(H).
\label{HVF2t}
\end{equation}
\end{Corollary}
\begin{proof}  We just need to compute
$$
\LD^\natural_T ( dx \wedge \theta^0_\sEL \cdot P) = dx \wedge \theta^0_\sEL \cdot P_t + dx \wedge \theta^0_\sEL \cdot \left( \bLP K-\bLP^* K \right) =  dx \wedge \theta^0_\sEL\cdot \left( P_t +   \bLP K-\bLP^* K \right)
$$
Using this computation in equation \ref{CLpsi2} with $\lambda =2 H dx$ gives equation \ref{HVF2t}.
\end{proof}
We call $H$ in \ref{HVF2t} the Hamiltonian.


\begin{Remark} A symplectic form $\Sigma$ is $t$-invariant if $\LD_{\partial_t} \Sigma = 0$. In this case $\Sigma$ determines a well defined symplectic form $\bar \Sigma$ on the quotient by the flow of $\partial_t$, $q:E \to E/\partial_t= \EJ$ such that $q^* \bar \Sigma = \Sigma$. Definition \ref{HVF0} where $Y$ and $\Sigma$ are $t$-invariant implies equation \ref{tHVF} for $\bar \Sigma$ and $\bar Y = q_* Y$.
\end{Remark}

\section{Variational and Symplectic Operator Equivalence}\label{Symplectic}

A time independent evolution equation $u_t=K(x,u,u_x,\dots, u_n) $ is Hamiltonian \cite{dorfman:1993a} if there exists a symplectic operator $\CS$ and a function $H$ called the Hamiltonian such that equation \ref{SymOp} holds. With this definition, the determination of the symplectic Hamiltonian equations is typically approached in two ways. The first way consists of determining the possible symplectic operators of a certain order \cite{dorfman:1993a}. Then for a given class of symplectic operators $\CS$, determine $K$ which satisfy equation \ref{SymOp}.  The second approach starts with a given $K$ and then determines if there exists a symplectic operator $\CS$ such that equation \ref{SymOp} holds. 

By comparison Theorem \ref{THC} combines these two questions and resolves the characterization of symplectic Hamiltonian evolution equations by the invariants $H^{1,2}(\CR)$. This simultaneously solves the existence of $\CS$ and the existence of the Hamiltonian function $H$ in \ref{SymOp}.

\subsection{$ H^{1,2}(\CR)$ and Symplectic Hamiltonian Evolution Equations. }

Given  a scalar evolution equation $u_t=K(t, x, u, u_x,\ldots)$, identify the manifolds $\CR$ and $E=\reals \times \EJ$ by identifying their coordinates which in turn induces an identification of smooth functions. Define the bundle map $\Pi:T^*(\CR) \to T^*(E)$ which is a projection operator by
\begin{equation}
\Pi(\theta^i) =  \theta_\sEL^i,\quad  \Pi(dx)  = dx, \quad \Pi(dt) = 0,
\label{bphi}
\end{equation}
where $\theta^i$ are given in equation \ref{DT} and $\theta^i_\sE$ are in equation \ref{CFE}.
  Also denote by $\Pi$ the induced projection map $\Pi:\Omega^{r,s}(\CR) \to \Omega_{\rm t_{sb}}^{r,s}(E)$ where
for example
\begin{equation}
\Pi\left(dx \wedge \theta^0 \wedge( s_i\theta^i)  + dt \wedge \beta \right) = dx \wedge \theta^0_\sEL \wedge(s_i \theta^i_\sEL).
\label{PSE}
\end{equation}

\begin{Lemma}  The map $\Pi:\Omega^{r,s}(\CR) \to \Omega_{\rm t_{sb}}^{r,s}(E) $ is a bicomplex homomorphism,
\begin{equation}
\Pi(\, d_\sH  \omega \,) = d^\sE_\sH \, \Pi(\omega), \qquad \Pi(d_\sV \omega) = d_\sV^\sE\, \Pi (\omega ) .
\label{PXC}
\end{equation}
\end{Lemma}
\begin{proof} Equation \ref{PXC} follows for the case $\omega=\theta^i$ directly from equations \ref{dHP2} and \ref{dHdVE},
and generically from the anti-derivation property of the operators. 
\end{proof}

\begin{Lemma} \label{phc} The function $\Pi : \Omega^{1,2}(\CR) \to  \Omega_{\rm t_{sb}}^{1,2}(E) $ induces a well defined injective linear map $\widehat \Pi: H^{1,2} (\CR) \to  \ker \delta^\sE_\sV\subset \CF_{\rm t_{sb}}^i(E) $ defined by
\begin{equation}
\widehat \Pi([\omega]) = I_\sEL \circ \Pi (\omega)
\label{HtoF}
\end{equation}
where $\omega$ is a representative of $[\omega]$. 
\end{Lemma}

\begin{proof} To show  $\widehat \Pi$ is well defined,  suppose  $\omega'=\omega + d_\sH\xi$. Then by equation \ref{PXC} and property 3 in equation \ref{IEP} applied to $I_\sEL$ gives
$$
I_\sEL \circ \Pi(\omega') =I_\sEL \circ (\Pi(\omega)+ \Pi(d_\sH \xi) ) =I_\sEL \circ (\Pi(\omega)+ d^\sE_\sH( \Pi(\xi) ) = I_\sEL \circ \Pi(\omega).
$$
Therefore $\widehat \Pi$ is well defined.

We now show $ \widehat \Pi([\omega])$ is $\delta^\sV_\sE$ closed. We use equation \ref{PXC} and compute
\begin{equation}
I_\sEL \circ d_\sV^\sE \circ I_\sEL \circ \Pi (\omega)=I_\sEL \circ d_\sV^\sE \circ \Pi (\omega)=
I_\sEL \circ \Pi (d_\sV \omega).
\label{ISF}
\end{equation}
Since $d_\sV \omega \in H^{1,3}(\CR)$,  Theorem \ref{ZC} implies there exists $\xi \in \Omega^{0,3}$ such that $d_\sV \omega = d_\sH \xi$ so equation \ref{ISF} becomes
$$
I_\sEL \circ d_\sV^\sE \circ I_\sEL \circ \Pi (\omega)=I_\sEL \circ  \Pi (d_\sH \xi)=I_\sEL \circ d_\sH^\sE \circ \Pi (\xi)
=0.
$$
Therefore  $\widehat \Pi([\omega])$ is $\delta_\sV$ closed. 

We now show $\widehat \Pi$ is injective. Let $[\omega]\in H^{1,2}(\CR)$ and let $\omega= dx \wedge \theta^0 \wedge \epsilon - dt \wedge \bbeta(\epsilon)$ be the unique representative from Theorem \ref{Urep}, where $\epsilon =r _i \theta^i$ and $\epsilon^* =-\epsilon$. Then $\widehat \Pi([\omega]) = dx \wedge \theta^0_\sEL\wedge(s_i\theta^i_\sEL)$, and $(s_i \theta^i_\sEL)^* = -(s_i \theta^i_\sEL)$ since $X=D_x$.  If $\widehat \Pi ([\omega])=0$, then $ s_i \theta^i_\sEL=0$
and $\omega=0$. This shows $\widehat \Pi([\omega]) \in \CF_{\rm t_{sb}}^2(E)$ and that $\widehat \Pi$ is injective.
\end{proof}

In particular we have 

\begin{Corollary} If $[\omega] \neq 0$ then $\widehat \Pi([\omega]) \in \CF_{\rm t_{sb}}^2(E)$ is a symplectic form.
\end{Corollary}

We now set out to prove the fact that $\widehat \Pi$ in Lemma \ref{phc} is in fact a bijection which will imply  Theorem \ref{TIS2} in the Introduction.  We will use the following Lemma.

\begin{Lemma} Let $s_i,\xi_{ij} \in C^\infty(\CR)$ then
\begin{equation}
\begin{aligned}
dt \wedge\left(  \pi^{1,2}\circ {\LD}_T  ( dx \wedge \theta^0_\sEL \wedge s_i \theta^i_\sEL) \right) &=
dt \wedge  {\LD}_T  ( dx \wedge \theta^0_\sEL \wedge s_i \theta^i_\sEL) = dt \wedge  \LD_T (dx \wedge \theta^0 \wedge  s_i \theta^i), \\
dt \wedge dx \wedge   D_x( \xi_{ij} \theta^i_\sEL \wedge \theta^j_\sEL) & = dt \wedge dx \wedge X( \xi_{ij} \theta^i \wedge \theta^j)
\end{aligned}
\label{FBE1}
\end{equation}
\end{Lemma}
\begin{proof} Since  $dt \wedge \theta^i_\sEL=dt \wedge \theta^i $ and $X=D_x$ these identities follow.
\end{proof}

We now have the main theorem.

\begin{Theorem} \label{keyT}  Let $\CS= s_i D_x^i$ be a skew-adjoint differential operator. The form  $\Sigma =dx \wedge \theta^0_\sEL \wedge (s_i \theta^i_\sEL) $ is symplectic, and  $Y= pr (K \partial_u)$ is a Hamiltonian vector-field for $\Sigma$ if and only if
\begin{equation}
\omega = dx \wedge \theta^0 \wedge \epsilon  - dt \wedge  \bbeta( \epsilon ) 
\label{keyEQ}
\end{equation}
satisfies $d_\sH \omega=0$, where $\epsilon = \CS (\theta^0)$.
\end{Theorem}
\begin{proof}  
Supposed $\Sigma$ is symplectic and $Y$ is Hamiltonian, then Lemma \ref{IT1} produces $\xi=\xi_{ab} \theta^a_\sEL\wedge \theta^b _\sEL$ satisfying equation \ref{dxi}.   Let 
\begin{equation}
\omega=  dx \wedge \theta^0 \wedge (s_i \theta^i) + dt \wedge ( \xi_{ab} \theta^a \wedge \theta^b).
\label{t1omega}
\end{equation}
Equations \ref{FBE1}  and  \ref{dxi} give
$$
\begin{aligned}
d_\sH \omega &= dt \wedge T(dx \wedge \theta^0 \wedge (s_i \theta^i)) +dx  \wedge X( dt\wedge \xi_{ab} \theta^a \wedge \theta^b )\\
& = dt  \wedge {\LD}_T  ( dx \wedge \theta^0_\sEL \wedge (s_i \theta^i_\sEL ) +dx \wedge dt \wedge  D_x( \xi_{ab} \theta^a_\sEL \wedge \theta^b_\sEL) \\
& = dt \wedge  \left( \pi^{1,2}\circ {\LD}_T  (dx \wedge \theta_\sEL^0 \wedge  (s_i \theta^i_\sEL ) ) - dx \wedge D_x( \xi_{ab} \theta_\sEL^a \wedge \theta_\sEL^b) \right) =0.
\end{aligned}
$$
Therefore $[\omega ]\in H^{1,2}(\CR)$.  Now Corollary \ref{CFbeta} implies $\xi_{ab} \theta^a \wedge \theta^b = - \bbeta( s_i \theta^i)$ so that $\omega $ in equation \ref{t1omega} and equation \ref{keyEQ} are the same.

Suppose now that $\omega$ in equation \ref{keyEQ} is $d_H$ closed. By Lemma \ref{phc}, $\Sigma = \widehat \Pi( \omega)$ is a symplectic form. So we need only show that $Y$ is Hamiltonian. Again we refer to Lemma \ref{IT1} and show the existence of $\xi=\xi_{ab} \theta^a_\sEL\wedge \theta^b_\sEL$ in equation \ref{dxi}.

Writing $\bbeta( s_i \theta^i)= B_{ab} \theta^a \wedge \theta^b$ and using equations \ref{FBE1} we have
\begin{equation}
\begin{aligned}
d_H \omega &= dt \wedge T(dx \wedge \theta^0 \wedge (s_i \theta^i)) -dx  \wedge X( dt\wedge B_{ab} \theta^a \wedge \theta^b)\\
&=dt \wedge \left( \pi^{1,2} \circ {\LD}_T  (dx \wedge \theta_\sEL^0 \wedge  (s_i \theta^i_\sEL ) ) + dx \wedge D_x(B_{ab} \theta_\sEL^a \wedge \theta_\sEL^b) \right).
\end{aligned}
\end{equation}
This will vanish if and only if
\begin{equation}
 \pi^{1,2} \circ {\LD}_T  (dx \wedge \theta_\sEL^0 \wedge  (s_i \theta^i_\sEL ) ) + dx \wedge D_x( B_{ab} \theta_\sEL^a \wedge \theta_\sEL^b)=0
\label{LISH}
\end{equation}
because this term is $t$ semi-basic. Equation \ref{LISH}  produces $\xi =-\Pi(\bbeta( s_i \theta^i))= B_{ab}\theta^a_\sEL \wedge \theta^b_\sEL$ in equation \ref{dxi} and therefore $Y$ is a Hamiltonian vector field for $\Sigma$.
\end{proof}

We now summarize the results by the following Theorem whose proof follows directly from Lemma \ref{phc} and Theorem \ref{keyT}

\begin{Theorem} \label{BST} Let $u_t=K$ be an evolution equation, and let $Y=pr(K\partial_u)$ be the evolutionary vector field on $E$ and let $\CZ_Y(E)\subset \CF_{\rm t_{sb}}^2(E)$ be the subset of symplectic forms for which $Y$ is a Hamiltonian vector field.  Define the function $\Psi:\CZ_Y(E)\to H^{1,2}(\CR)$ given by
\begin{equation}
\Psi(dx \wedge \theta^0_\sEL\wedge \epsilon_\sEL ) = [dx \wedge \theta^0 \wedge \epsilon-  dt \wedge  \bbeta(\epsilon)],
\label{TY}
\end{equation}
where $ dx \wedge \theta^i_\sEL \wedge \epsilon_\sEL \in \CZ_Y(E)$ with $\epsilon_\sEL=\CS(\theta^0_\sEL)$ and $\CS=s_i D_x^i$ is the corresponding symplectic operator, and $\epsilon =\CS(\theta^0)=s_iX^i(\theta^0)$. The function $\Psi:\CZ_Y(E)\to H^{1,2}(\CR)$ is an isomorphism and $ \hat \Psi =\Psi^{-1}$ where $\hat \Psi$ is defined in equation \ref{HtoF}.
\end{Theorem}

With Theorems \ref{THA} and \ref{BST} in hand the proof of Theorem \ref{TIS2} and \ref{THC} are now easily given.
\begin{proof} (Theorems \ref{TIS2} and \ref{THC})
  Suppose that $\CS= s_i D_x^i\in \CZ_Y(E)$ is a symplectic operator for
the scalar evolution equation $\Delta=u_t-K$ and that $Y=\pr( K\partial_u)$ is a Hamiltonian vector field for $\Sigma$. Then with $\Psi$ from equation \ref{TY} in  Theorem \ref{BST} and $\Phi$ in equation \ref{FISO} in Theorem \ref{THA} we have,
$$
\Phi^{-1}\circ \Psi(\CS)= - \frac{1}{2} \ts_i D_x^i
$$
is a variational operator and so $\CS$ is a variational operator for $\Delta$ (by the abuse of notation in Remark \ref{thm_rmk}). The fact that $\Phi^{-1}\circ \Psi$ is an isomorphism then proves Theorem \ref{TIS2}.

As above we identify a symplectic operator $\CS$ on $E$ as an operator on $\CJ$, then the function $\Phi$ in equation \ref{FISO} defines an isomorphism between symplectic operators for $\Delta$ and $H^{1,2}(\CR)$. This proves Theorem \ref{THC}.
\end{proof}

As our final Lemma we show for completeness how formula \ref{thm_m} can be determined from the symplectic potential.

\begin{Lemma} Let $\CS=s_i D_x^i$ be a symplectic operator and
let $\psi = dx\wedge \theta^0_E \cdot P\in C^\infty(E)$ be a symplectic potential.  The unique representative for $\Psi(\CS)\in H^{1,2}(\CR)$ in Theorem \ref{rho_e} 
has $\epsilon  = \CS( \theta^0)$. Furthermore there exists a representative $\omega$ for $\Psi(\CS)$ where $\omega $ in equation \ref{conf_2e} can be written $\omega = d_\sV \eta$ where
\begin{equation}
\eta = dx \wedge \theta^0 \cdot  P - dt \wedge \gamma.
\label{etaP}
\end{equation}
 \end{Lemma}

\begin{proof} By equation \ref{TY} of Theorem \ref{BST} we have the unique representative as stated in the Lemma. 

To prove the second part of the lemma by using Theorem \ref{TCF3} to construct  a representative $\omega_0$ for $\Psi(\CS)$ such that $\omega_0= d_\sV \eta_0$ with $\eta_0= dx \wedge \theta^0 \cdot  Q - dt \wedge \gamma_0$. 

By equation  \ref{epsQ} of  Corollary \ref{Curep2}  and equation \ref{SinP} for the operator in the form $\epsilon=\CS(\theta^0)$ gives
\begin{equation}
\epsilon  = \frac{1}{2}({\bf L}_Q- {\bf L}_Q^*) \theta^0= \frac{1}{2}({\bf L}_P- {\bf L}_P^*) \theta^0.
\label{epseq}
\end{equation}
Lemma \ref{SyminP} and equation \ref{epseq} show $\psi_0 = dx \wedge \theta^0_\sEL \cdot Q$ is a symplectic potential for $\CS$ and that  $\delta_\sV^\sE \psi_0= \delta_\sV^\sE \psi$. Therefore using equation \ref{IEP} (for $I_\sE$) and
the exactness of the $d_\sV^\sE$ complex, 
\begin{equation}
\psi =\psi_0 +d^\sE_\sV (A dx) +  d_\sH^\sE\, \xi 
\label{psipsi0}
\end{equation}
for some $A\in C^\infty(E)$ and $\xi \in \Omega^{0,1}(E)$.  We then let
\begin{equation}
\eta= \eta_0 + d_\sV (A dx) + d_\sH \xi, \quad {\rm and} \quad
\omega = \omega_0 - d_\sH d_\sV \xi,
\label{etaomega}
\end{equation}
where we are computing $d_\sH$ and  $d_\sV $ on $\CR$. Note that by equation \ref{PXC} and \ref{psipsi0} we have $\Pi(\eta) = \psi $
so that $\eta$ has the form in equation \ref{etaP}. We then compute using equation \ref{etaomega}
$$
d_\sV \eta= d_\sV(\eta_0+d_\sH \xi ) = \omega_0 +d_\sV d_\sH \xi = \omega,
$$
which proves the lemma.
\end{proof}

\subsection{Time Independent Operators}

Equation \ref{SymOp} defines when the time independent evolution equation $u_t=K(x,u,u_x,\ldots)$ is a Hamiltonian system with symplectic operator $\CS$. This is precisely the same definition that the ordinary differential equation $K(x,u,u_x,\ldots)=0$ admits a variational operator. The following simple lemma is the key to decoupling the variational operator problem 
for time independent  scalar evolution equations.  

\begin{Lemma} \label{CSDL} Let $ \CS=s_i D_x^i$ be a time independent symplectic operator with symplectic potential $P \in C^\infty(J^\infty(\reals,\reals))$ (equation \ref{CF31} in Lemma \ref{SymPot}).  Then 
\begin{equation}
\CS(u_t) = \Eop \left(- \frac{1}{2} \tP u_t \right).
\label{CSPu}
\end{equation}
\end{Lemma}
\begin{proof} 
By the product formula in the calculus of variations (equation 5.80 in \cite{Olver:1993a}) the left side of equation \ref{CSPu} is
\begin{equation}
\begin{aligned}
\Eop \left(- \frac{1}{2} \tP u_t \right)& =-\frac{1}{2} \left( {\bf F}_{\tP}^* u_t + {\bf F}^*_{u_t} \tP \right) \\
& = -\frac{1}{2} \left( {\bf F}_{\tP}^* u_t - D_t \tP \right) \\
& = -\frac{1}{2} \left( {\bf F}_{\tP}^* u_t - \tP_i D_x^i u_t \right)\\
& = -\frac{1}{2} \left( {\bf F}_{\tP}^* u_t - {\bf F}_P  u_t \right).
\end{aligned}
\label{CSpu2}
\end{equation}
Equation \ref{CSpu2} together with the fact from equation \ref{CF31} $2\CS(u_t) = {\bf F}_{\tP}u_t - {\bf F}_{\tP}^*u_t$ show that the two sides of equation \ref{CSPu} agree. 
\end{proof} 

We then have the following.

\begin{Theorem} \label{TFAE} Let $\CS$ be a $t$-independent symplectic operator. The following are equivalent,
\begin{enumerate}

\item $u_t=K(x,u,u_x,\ldots)$ is Hamiltonian the sense of $\CS(K)=\Eop(H)$.

\item $\CS$ is a symplectic variational operator for the ODE $K=0$,

\item $\CS$ is a variational operator for $u_t=K$ (see Remark \ref{thm_rmk}).

\end{enumerate}

\end{Theorem}

This converts the symplectic Hamiltonian question for the evolution equation into a variational operator problem for the ODE $K=0$, see \cite{fels:2018a}.

\begin{proof}  Suppose $u_t = K(x,u,...,u_n)$  is Hamiltonian for the $t$-independent symplectic operator $\CS$, so that $\CS(K) = \Eop(H)$ on $J^\infty(\reals,\reals)$. Therefore  $\CS$ is a variational operator for the ODE $K=0$. So (1) and (2) are trivially equivalent. 

We show (1) implies (3). Suppose that $\CS(K) = \Eop(H)$. Using equation \ref{CSpu2} in Lemma \ref{CSDL} we have
$$
\CS (u_t-K) =\CS(u_t)- \Eop(\tH) = \Eop\left( -\frac{1}{2} \tP u_t-\tH \right).
$$
Therefore $\CS$ is a variational operator for $u_t-K$.

Finally we show (3) implies (1). Starting with hypothesis (3) in the form of equation \ref{thm_m} we have,
\begin{equation}
({\bf F}_{\tP}^* -{\bf F}_{\tP})(u_t -K )= \Eop( \Delta \tP +\tL)= \Eop(\tP u_t-\tP K +\tL),
\label{EQE1}
\end{equation}
where $\tL = L \circ \pi$ (see \ref{thm_rmk}). Substituting from equation \ref{CSpu2} into equation \ref{EQE1} we get 
$$
({\bf F}_{\tP}^* -{\bf F}_{\tP})(-K )= \Eop(-\tP K +\tL).
$$
Therefore
$$
\CS (K)  = \Eop( 2 (L - PK)),
$$
and $u_t=K$ is a time independent Hamiltonian evolution equation for the symplectic operator $\CS$.
\end{proof}


\section{First Order Operators}\label{FirstOrder}

For a third order evolution equation 
\begin{equation}
u_t=K(t,x,u,u_x,u_{xx},u_{xxx})
\label{TOE}
\end{equation}
we write the conditions $\theta^0 \wedge \CL_\Delta^*(\epsilon) = 0$, when $\epsilon$ is first order and skew-adjoint. This will prove Theorem \ref{FOT} in the Introduction.


\begin{proof} (Theorem \ref{FOT}) By Theorems \ref{rho_e} and \ref{Urep}  the skew-adjoint operator $\CE=2 \tR D_x + X(\tR)$ is a variational operator for \ref{TOE} if and only if
the  skew-adjoint form $\epsilon= -R\theta^1-\frac{1}{2}R_0\theta^0$ is a solution to 
\begin{equation}
-\CL_\Delta^*(\epsilon)\wedge \theta^0 = \left(T(\epsilon)  -X^3(K_3 \epsilon) +X^2(K_2 \epsilon) -X(K_1\epsilon) +K_0\epsilon \right)\wedge \theta^0=0
\label{findrho}
\end{equation}
where $K_i = \partial_{u_i} K$. Using $T(\theta^0) = d_\sV K = K_i \theta^i$ and $T(\theta^1) = X( d_\sV K )=X( K_i \theta^i)$ we have
\begin{equation}
T(\epsilon) = -T( R )  \theta^1 -\frac{1}{2}T(X(R))) \theta^0 - RX( K_i \theta^i) -\frac{1}{2} X(R) K_i \theta^i.  
\label{Trho}
\end{equation}
The highest possible $\theta^i \wedge \theta^0 $ term in equation \ref{findrho} using \ref{Trho} is $\theta^4$. We find from equation \ref{findrho}
$$
[ \theta^4\wedge \theta^0]= -RK_3  + RK_3 =0.
$$
While for $\theta^3\wedge \theta^0$, $\theta^2\wedge \theta^0$ and $\theta^1\wedge \theta^0$ we have from equations \ref{Trho} and \ref{findrho},
\vspace{-0.15in}
\begin{equation}
\label{lt3}
\begin{aligned}
\quad\\
[\theta^3\wedge \theta^0]=& 2X(K_3) R -2K_2 R +3K_3 X( R ), \\
[\theta^2\wedge \theta^0]=&-3X(K_2R)+3X^2(RK_3) +\frac{3}{2} X(K_3 X( R )) , \\
[\theta^1\wedge \theta^0]=&- T(R) - 2 K_0 R +K_1X(R)  +\frac{3}{2} X^2(K_3 X(R ))\\ &+X^3(K_3 R)-X^2(K_2 R)-X(K_2 X(R)).
\end{aligned}
\end{equation}

For the coefficient of $\theta^3\wedge \theta^0$ to be zero we have from equation \ref{lt3},
\begin{equation}
X(R) = \frac{2}{3K_{3}}(K_{2} - X(K_{3})) R = \hat K_2 R
\label{XR}
\end{equation}
where $\hat K_2 =  \frac{2}{3K_{3}}\left(K_2 - X(K_3)\right)$.
The coefficient of $\theta^2\wedge \theta^0$ in equation \ref{lt3} is zero on account of \ref{XR}. 
For the coefficient of $\theta^1\wedge \theta^0$ in \ref{lt3} to be zero gives
\begin{equation}
T(R) = - 2 K_0 R +K_1X(R)  +\frac{3}{2} X^2(K_3 X(R))+X^3(K_3 R)-X^2(K_2 R)-X(K_2 X(R)).
\label{FoT}
\end{equation}
Simplifying equation \ref{FoT} using equation \ref{XR}  we get
\begin{equation}
\begin{aligned}
T(R)   = \left(
-2K_0+  K_1\hat K_2 -\frac{1}{2}\left(X(K_3){\hat K_2^2} +K_3 \hat K_2^3\right)+X( K_3 X(\hat K_2)) \right) R
\end{aligned}
\label{TRF}
\end{equation}
It follows that a non-vanishing $R$ (which we may assume to be positive) satisfying equations \ref{XR} and \ref{TRF} is necessary and sufficient for the existence of a first order variational operator for $\Delta=u_t-K$ in equation \ref{TOE} is equivalent to $A=X(\log R)$ and $B=T(\log R)$ satisfying the conditions in Theorem  \ref{FOT}. This proves Theorem \ref{FOT}.
\end{proof}

For the KdV equation $u_t = u_{xxx}+ u u_x$  the form $\kappa$ in equation \ref{kappadef} is
$$
\kappa = -u_x dt, \quad d_\sH \kappa= -u_{xx} dx \wedge dt.
$$
Therefore according to Theorem \ref{FOT} there is no first order symplectic formulation for the KdV equation as a Hamiltonian evolution equation.

\subsection{First order Hamiltonians Operators and Bi-Hamiltonian Evolution equations}

Let $ v_t= \CD \circ \Eop(H(x,v,v_x,\ldots))$ be a Hamiltonian evolution equation where $\CD$ is a first order Hamiltonian operator. According to \cite{Olver:1988a} or \cite{vino:1986a} we may choose coordinates (using a contact transformation) such that $\CD= D_x$.  The following is Theorem 1 in \cite{Nutku:2002a} in the context of scalar evolution equations.

\begin{Lemma}  \label{PFT} The potential form of the Hamiltonian evolution equation,
\begin{equation}
v_t = D_x \Eop(H_1(x,v,v_x,\ldots)).
\label{CFOH}
\end{equation}
is given by the equation
\begin{equation}
u_t = \Eop(H_1)|_{v=u_x}.
\label{pf2}
\end{equation}
The potential form \ref{pf2} admits $\CE=D_x$ as a first order variational operator, and satisfies
\begin{equation}
D_x\left( u_t - \Eop(H_1)|_{v=u_x} \right) = \Eop( -\frac{1}{2}{u_x}{u_t} + H_1|_{v=u_x} ).
\label{vop1}
\end{equation}
\end{Lemma}

There is an abuse of notation in this lemma where $D_x$ is used as the total $x$ derivative operator in either variable $u$ or $v$ depending on context.

\begin{proof}   Starting with equation  \ref{CFOH},   let $v=u_x$ so that \ref{CFOH} becomes
\begin{equation}
u_{tx} = \left(D_x \Eop(H_1) \right) \vert_{v=u_x}=  D_x \left( \Eop(H_1)\vert_{v=u_x}\right).
\label{ppf}
\end{equation}
Integrating equation \ref{ppf} with respect to $x$ gives the potential form \ref{pf2}.

To prove equation \ref{vop1} holds we simply need the change of variables formula, see exercise 5.49 in \cite{Olver:1993a},
\begin{equation}
\Eop\left(  H_1|_{v=u_x}\right) = (D_x)^*\left( \Eop(H_1|_{v=u_x}) \right)= -D_x \left(\Eop(H_1|_{v=u_x})\right).
\label{CVCV}
\end{equation}
Equation \ref{CVCV} together with the simple fact $-2 \Eop(u_tu_x) = u_{tx}$ proves equation \ref{vop1}.
\end{proof}

The second term in the right hand side of equation \ref{vop1}  is just the pullback of the Hamiltonian function in \ref{CFOH}. We also note the following simple corollary.

\begin{Corollary} Every  Hamiltonian evolution equation $v_t=\CD (\Eop (H_1(x,v,v_x,\ldots)))$ with first order Hamiltonian operator $\CD$ is the symmetry reduction of an equation $u_t= K(x,u, u_x,\ldots)$, of the same order, which admits an invariant first order variational operator. 
\end{Corollary}

\subsection{Bi-Hamiltonian Evolution Equations with a First Order Hamiltonian Operator}

We now present sufficient conditions when the potential form of a compatible bi-Hamiltonian system admits a second variational operator. 

\begin{Theorem} \label{biht}  Let $v_{t}=K(x,v,v_x,\ldots)=D_x\left( \Eop(H_1(x,v,v_x,\ldots)) \right) $ be a Hamiltonian evolution equation
with potential form
\begin{equation}
u_t = \Eop(H_1) \vert_{v=u_x}.
\label{pe2}
\end{equation}
Let $\CD_0$ be  second time independent Hamiltonian operator with Hamiltonian $H_0(x,v,v_x,\ldots)$ satisfying,
$$
v_t =D_x\left( \Eop(H_1) \right)= \CD_0  (\Eop(H_0)).
$$
Assume $\CD_0$ also satisfies the compatibility condition (equation 7.29 in \cite{Olver:1993a})
\begin{equation}
\CD_0 \left(\Eop( H_1 )\right) = D_x \Eop(H_2).
\label{cc1}
\end{equation}
Then the right hand side of the potential form satisfies 
\begin{equation}
\CE (\Eop(H_1) \vert_{v=u_x} )=- \Eop ( H_2\vert_{v=u_x})
\label{EKT}
\end{equation} 
where $\CE = \CD_0|_{v=u_x}$\footnote{ $\CD_0$ is the push-forward of $\CE$ by the quotient map ${\bf q}:(t,x,u,u_x,\ldots)\to (t, x, v, v_x,\ldots)$.}.  
Furthermore if $\CE=\CD_0|_{v=u_x}$ is symplectic,  then $\CE$ is a variational operator for the evolution equation \ref{pe2} and
\begin{equation}
\CE( u_t -K) = \Eop( Q u_t + H_2\vert_{v=u_x})
\label{CEH2}
\end{equation}
where $Q$ is defined in equation  \ref{CEEQS} where $\CE= F_Q^* - F_Q$.
\end{Theorem}

\begin{proof}
First we apply $\CE=\CD_0|_{v=u_x}$ to the right hand side of equation \ref{pe2}, and use condition \ref{cc1} to get
\begin{equation}
\begin{aligned}
\CE( \Eop(H_1) \vert_{v=u_x} ) & =  \left(\CD( \Eop(H_1) \right)\vert_{v=u_x} \\
& = \left( D_x (\Eop( H_2) \right)\vert_{v=u_x}\\
& = - \Eop( H_2\vert_{v=u_x}).
\end{aligned}
\label{phv1}
\end{equation}
Again the last line follows from the change of variables formula in the calculus variations (exercise 5.49 in \cite{Olver:1993a}). This verifies equation \ref{EKT}.  Then by part (1) of Theorem \ref{TFAE} equation \ref{EKT} shows that $\CE$ is a variational operator for equation \ref{pe2}. If $Q$ is the function from equation \ref{CEEQS} we then have $\CE(u_t)=(F_Q^* - F_Q)(u_t)=E(Qu_t)$ and equation \ref{phv1} that
\begin{equation}
\CE(u_t  -  \Eop(H_1) \vert_{v=u_x}) =\Eop\left( Q u_t +H_2\vert_{v=u_x}\right).
\label{QH}
\end{equation}
\end{proof}

Theorem \ref{biht} makes the hypothesis that $\CE=\CD|_{v=u_x}$ is a symplectic operator. This holds  in the case of the Hamiltonian operators given by Theorem 5.3 in \cite{dorfman:1993a},
\begin{equation}
\CD = h(v)\left( \sqrt{ c_1 + c_2 \int_0^v \frac{1}{h(y)}dy} D_x \circ  \sqrt{ c_1 + c_2 \int_0^v \frac{1}{h(y)}dy} + D_x^3\right) \circ h(v)
\label{CDdorf}
\end{equation}
satisfy the compatibility conditions with $D_x$ in Corollary 3.2 of \cite{cooke:1991a} when $h(v) = (k_1 v+k_2)^{-1}$. This gives
\begin{equation}
\CE=\CD|_{v=u_x} =  \frac{1}{k_1u_x+k_2}\left(\sqrt{c_1+\frac{1}{2}k_1u_x^2+k_2 u_x} D_x \circ \sqrt{c_1+\frac{1}{2}k_1u_x^2+k_2 u_x}  + D_x^3\right) \circ \frac{1}{k_1u_x+k_2}
\label{comsym}
\end{equation}
which are symplectic \cite{dorfman:1993a}.

\section{Examples}

\begin{Example} \label{HD1} The Harry-Dym equation $ z_t = z^3 z_{xxx} $ \cite{dorfman:1993a, wang:2002a} is a compatible bi-Hamiltonian system,
\begin{equation}
z_t = \widehat \CD_1 \Eop( \widehat H_1)= \widehat \CD_0 \Eop( \widehat H_0)
\label{HDZ}
\end{equation}
where
\begin{equation}
\widehat \CD_1=  {z^2}\circ D_x \circ {z^2} \ , \  \widehat H_1  = -\frac{1}{2}\frac{z_x^2}{z} \ , \ {\rm and} \quad  \widehat \CD_0  = z^3\circ  D_x^3 \circ  z^3 \ , \ \widehat H_0 = -\frac{1}{z}.
\label{HDOPSZ}
\end{equation}

The change of variable $z=v^{-1}$ maps the Hamiltonian operator $\widehat \CD_1$ to canonical form  \cite{vino:1986a,Olver:1988a},  and the Hamilonian operators and the associated Hamiltonians in equation \ref{HDOPSZ} become
\begin{equation}
\CD_1=  D_x \ , \  H_1  = -\frac{1}{2}\frac{v_x^2}{v^3} , \quad {\rm and} \quad
\CD_0=  v^{-1}D_x^3\circ v^{-1} \ , \  H_0  = - v.
\label{HDv}
\end{equation}
The Harry-Dym equation \ref{HDZ} in these coordinates is then,
\begin{equation}
v_t= D_x \Eop( H_1)=   \CD_0 (\Eop(H_0))=\frac{v_{xxx}}{v^3}-6\frac{v_x v_{xx}}{v^4}+\frac{6v_x^3}{v^5}.
\label{HDw}
\end{equation}

The potential form of equation \ref{HDw} is found by letting $v=w_x$  and integrating to get (see also \ref{pf2}) 
\begin{equation}
w_t= E(H_1)_{v=w_x}=\frac{w_{xxx}}{w_x^3}-\frac{3w_{xx}^2}{2 w_x^4}.
\label{PHD}
\end{equation}
Equation \ref{vop1} of Lemma \ref{PFT}  as it applies to the potential Harry-Dym equation \ref{PHD}  produces the following variational operator equation for $D_x$,
$$
D_x\left( w_t - \frac{w_{xxx}}{w_x^3}+\frac{3w_{xx}^2}{2 w_x^4}\right)=
\Eop\left(  -\frac{1}{2} w_x w_t -\frac{1}{2}\frac{w_{xx}^2}{w_x^3} \right) .
$$

We now apply Theorem  \ref{biht} to obtain a second variational operator. The compatibility condition in equation \ref{cc1} is satisfied with the operators from equation \ref{HDv} with
\begin{equation}
\CD_0( \Eop( H_1)) = D_x \Eop( H_2),\quad {\rm where} \quad H_2= \frac{1}{2}\frac{v_{xx}^2}{v^5} -\frac{15}{8} \frac{v_x^4}{z^7}.
\label{CCHD1}
\end{equation}
The operator $\CD_0$ in equation \ref{HDv} is of the form \ref{CDdorf} so that by equation \ref{comsym}
\begin{equation}
\CE=\CD_0|_{v=w_x}= w_x^{-1}D_x^3\circ w_x^{-1} 
\label{HDCE}
\end{equation}
is a symplectic or variational operator. Since the compatibility  condition in equation \ref{cc1} is satisfied and $\CE$ is a symplectic operator Theorem \ref{biht} applies.  The operator $\CE$ in equation \ref{HDCE} is a variational operator  for the potential Harry-Dym equation in \ref{PHD}. The function $Q$ in equation \ref{CEEQS} is easily determined for $\CE$ (using the fact that $-2\CE$ is a symplectic operator) to be 
\begin{equation}
Q=\frac{w_{xx}^2-w_x w_{xxx}}{2w_x^{3}}.
\label{HDQ}
\end{equation}
Equation \ref{CEH2} with $Q$ in equation \ref{HDQ} and
$H_2$ in equation \ref{CCHD1} (with $v=w_x$) gives the
variational operator equation for the potential Harry-Dym equation \ref{PHD},
$$
\CE \left( w_t -\frac{w_{xxx}}{w_x^3}+\frac{3w_{xx}^2}{2 w_x^4} \right) = 
\Eop\left(  \frac{w_{xx}^2-w_x w_{xxx}}{2w_x^{3}}w_t+\frac{1}{2}\frac{w_{xxx}^2}{w_x^5} -\frac{15}{8} \frac{w_{xx}^4}{w_x^7}\right).
$$
\smallskip

If we return to the original coordinates for the Harry-Dym equation and make the change of variable given by $ x=u, w=u_x, w_x=u_{xx}u_x^{-1}, \ldots $ to the potential form in equation \ref{PHD} we get the Schwarzian KdV (or Krichever-Novikov) equation (pg. 120 in \cite{dorfman:1993a}),
\begin{equation}
u_t = u_{xxx} -\frac{3}{2} \frac{u_{xx}^2}{u_x}.
\label{SKDV}
\end{equation}
In particular the Schwarzian KdV in equation \ref{SKDV} is the potential form of the Harry-Dym equation \ref{HDZ}. These different coordinate representations of the Harry-Dym and the Schwarzian KdV is summarized by the diagram,
\begin{equation}
\begin{gathered}
\begindc{\commdiag}[30]
\obj(0, 30)[KN]{$u_t = u_{xxx}  -\frac{3}{2} \frac{u_{xx}^2}{u_x}$}
\obj(80, 30)[PKN]{$w_t=\frac{w_{xxx}}{w_x^3}-\frac{3w_{xx}^2}{2 w_x^4} 
 $}
\obj(0, 0)[HD]{$ z_t = z^3 z_{xxx}$\,}
\obj(80, 0)[PHD]{$v_t=- v^{-1} D_x^3(v^{-1}) $}
\mor{KN}{HD}{$(x=u,z=u_x,z_x=u_{xx}u_x^{-1})$}[\atleft, \solidarrow]
\mor{KN}{PKN}{$(x=w,v=x, w_x=u_x^{-1})$}[\atleft, \solidline]
\mor{HD}{PHD}{$(x=x,v=z^{-1})$}[\atright, \solidline]
\mor{PKN}{PHD}{$(x=x,v=w_x)$}[\atleft, \solidarrow]
\enddc
\end{gathered}
\end{equation}

The variational or symplectic operators for the Schwarzian KdV are obtained by applying the change of variables $ x=u, w=u_x, w_x=u_{xx}u_x^{-1},\ldots$ to $D_x$ and equation \ref{HDCE} giving the well known symplectic or variational operators for the Schwarzian KdV \cite{dorfman:1993a},
\begin{equation}
\CE_1  = u_x ^{-1} D_x \circ u_x^{-1}=\frac{1}{u_x^{2}} D_x -\frac{u_{xx}}{u_x^3}, \quad \CE_0= \frac{1}{u_x^2} D_x^3- 3\frac{u_{xx}}{u_x^3} D_x^2 +\left(3\frac{u_{xx}^2}{u_x^4} -\frac{u_{xxx}}{u_x^3} \right) D_x.
\label{CHSKdV}
\end{equation}
With quotient map ${\bf q}(t, x, u, u_x,u_{xx}, \ldots) =(t=t,x=u,z=u_x,z_x=u_{xx}u_x^{-1},\ldots)$, 
the operators from \ref{CHSKdV} project ${\bf q}_* \CE_i = \tilde \CD_i$ to the Hamiltonian operators in equation \ref{HDOPSZ}.


We now compute the explicit unique representative for the $H^{1,2}(\CR)$ cohomology class 
 for the Schwarzian-KdV \ref{SKDV} corresponding to the 
first operator in \ref{CHSKdV} (Theorem \ref{Urep}). This is computed using formula \ref{omegad} in Theorem \ref{THA} to be,
\begin{equation}
\omega_1 = -\frac{1}{2u_x^2} dx \wedge \theta^0 \wedge \theta^1
+dt \wedge\left[ \theta^0 \wedge \left(
\frac{4u_{xxx}u_x-3u_{xx}^2}{4u_x^4} \theta^1 +
\frac{u_{xx}}{2u_x^3}  \theta^2- \frac{1}{2u_x^2}\theta^3\right) 
+ \frac{1}{u_x^2}  \theta^1 \wedge \theta^2 \right].
\end{equation}
We have $d_\sV \omega_1=0$ and for the forms $\eta$ and $\lambda$ in Theorem \ref{conf_2}  we may choose
\begin{equation}
\begin{aligned}
\eta_1&=\frac{1}{2u_x} dx \wedge \theta^0 + dt \wedge \left(\frac{u_{xx}^2-2u_{xxx}u_x}{4u_x^3}\theta^0+ \frac{u_{xx}}{2 u_x^2} \theta^1+\frac{1}{2u_x}\theta^2  \right)\\
\lambda_1&=-\frac{3u_{xx}^2}{4u_x^2} dt \wedge dx.
\end{aligned}
\label{lam1}
\end{equation}
Likewise formula \ref{omegad} for the second operator in \ref{CHSKdV} gives the unique cohomology representative ((Theorem \ref{Urep}),
\begin{equation}
\begin{aligned}
\hat \omega_0  &= dx \wedge \theta^0 \wedge\left( \frac{u_x u_{xxx} -3 u_{xx}^2}{2u_x^4} \theta^1 +\frac{3u_{xx}}{2u_x^3} \theta^2 - \frac{1}{2u_x^2}\theta^3\right) 
-\frac{1}{2 u_x^2} dt \wedge \theta^2 \wedge \theta^3  \\
& 
\quad 
+dt \wedge\theta^1 \wedge \left(\frac{1}{2u_x^2} \theta^4 -\frac{u_{xx}}{u_x^3}\theta^3  -\frac{5u_x u_{xxx}-6u_{xx}^2}{2u_x^4}\theta^2 \right)- \frac{1}{2u_x^2} dt \wedge \theta^0\wedge \theta^5 \\
&\quad
+\frac{2u_x^3u_{xxxxx}-18u_x^2u_{xx}u_{xxxx}-12u_x^2u_{xxx}^2+69u_x u_{xx}^2u_{xxx}-39u_{xx}^4}{4 u_x^6} dt \wedge \theta^0 \wedge \theta^1\\
&
\quad + dt \wedge \theta^0 \left( 
\frac{10u_x^2u_{xxxx}-48u_x u_{xx}u_{xxx}+39u_{xx}^3}{4u_x^5}\theta^2+
\frac{3(4u_x u_{xxx}-7u_{xx}^2)}{4u_x^4}\theta^3 +\frac{2u_{xx}}{ u_x^3}\theta^4  \right).
\end{aligned}
\end{equation}
In this case $d_\sV \omega_0\neq 0$, but $[\hat \omega_0]=[\omega_0]$ where
\begin{equation}
\omega_0 =\hat \omega^0 + d_\sH \left( \frac{u_{xx}}{2u_x^3} \theta^0  \wedge \theta^1\right)
\label{clom3}
\end{equation}
and $d_\sV \omega_0= 0$.  Futhermore with $\omega_0$ in equation \ref{clom3} the forms $\eta$ and $\lambda$ 
in Theorem \ref{conf_2} can be chosen to be
\begin{equation}
\begin{aligned}
\eta_0 & =\frac{2u_{xx}^2-u_x u_{xxx}}{2u_x^3} dx \wedge \theta^0 
+\frac{2u_{xx}^2-u_x u_{xxx}}{2u_x^3} dt \wedge \theta^2 +\frac{u_{xxxx} u_x^2-3 u_x u_{xx} u_{xxx}+u_{xx}^3}{2 u_x^4}  dt \wedge \theta^1\\
&\quad -\frac{2u_x^3u_{xxxxx}-10u_x^2u_{xx}u_{xxxx}-6u_x^2u_{xxx}^2+27u_x u_{xx}^2u_{xxx}-12u_{xx}^4}{4u_x^5} dt \wedge \theta^0\\
\lambda_0&=-\frac{u_{xx}^4}{8u_x^4} dt \wedge dx.
\end{aligned}
\label{L3SKdV}
\end{equation}
For $\lambda_i$ in equations \ref{lam1} and \ref{L3SKdV}, it is difficult to determine whether $[\lambda_i]\in H^{2,0}(\CR)$ is trivial or not (see Theorem \ref{kerPi}).  However, it is possible but not easy to show $\lambda_i \neq d \kappa_i$ where $\kappa_i$ is $t$-invariant by using the infinite sequence of conservation laws \cite{dorfman:1993a} for the Krichever-Novikov (Schwarzian KdV) equation \ref{SKDV}. The forms $\lambda_i$ define a non-trivial cohomology class in the $t$-invariant variational bi-complex for \ref{SKDV}.

\end{Example}

\begin{Example}  \label{HD2}  The Harry Dym equation can be written in the form
\begin{equation}
v_t = D_x^3\left( \frac{1}{\sqrt{v}}\right)= \CD_i \left(\Eop( H_i ) \right), \qquad i=0,1
\label{HDD3}
\end{equation}
where the Hamilonian operators and their Hamiltonians are
$$
\CD_0=  D_x^3 \ , \  H_0  = 2 \sqrt{v}\quad {\rm and} \quad
\CD_1= 2v D_x +v_x\ , \  H_1  = \frac{1}{8}v^{-\frac{5}{2} }v_x^2.
$$
Equation \ref{HDD3} is obtained from equation \ref{HDw} by substituting $v= -2^{\frac{1}{3}} \sqrt{\hat v}$.

Another {\it potential form} (or integrable extension) for the Harry-Dym equation \ref{HDD3} can be obtained  by letting $v=u_{xxx}$ in equation \ref{HDD3} so that
$$
u_{txxx} =\left( D_x^3  \Eop( H_0) \right) \biggr\rvert_{v=u_{xxx}},
$$
which after integrating three times gives,
\begin{equation}
u_t =  \Eop( H_0) \rvert_{v=u_{xxx}}= \sqrt{ \frac{1}{u_{xxx}} }.
\label{HD3PF}
\end{equation}
We show that $D_x^3$ is a variational operator. First using the change of variables formula in the calculus of variation for $v=u_{xxx}$ (exercise 5.49 \cite{Olver:1993a}) we have
\begin{equation}
\left(-D_x^3 \Eop( H_0)\right) \biggr\rvert_{v=u_{xxx}} = \Eop\left( H_0\vert_{v=u_{xxx}} \right).
\label{vop3hd3}
\end{equation}
The operator $D_x^3$ is symplectic which together with equation \ref{vop3hd3} shows that $D_x^3$ is a variational operator for equation \ref{HD3PF} and giving,
$$
D_x^3\left( u_t -  \Eop( H_0)\vert_{v=u_{xxx}}   \right) = 
u_{txxx} + \Eop\left(H_0\vert_{v=u_{xxx}}\right)= \Eop\left( -\frac{1}{2} u_t u_{xxx}+ 2\sqrt{u_{xxx}}\right).
$$

In equation \ref{cc1} compatibility was used to show the second Hamiltonian operator for a bi-Hamiltonian equation became a variational operator for the potential form. In order to use a similar argument in this case we need to show $\CD_1 \Eop(H_0) = \CD_0 \Eop( H_{-1})$. We find
\begin{equation}
\CD_1\left(\Eop (H_0)\right)
=( 2 vD_x +v_x)\left(\frac{1}{\sqrt{v}}\right) = 0 = \CD_0( 0 ).
\label{CCCH3}
\end{equation}
In analogy to equation \ref{cc1}, this gives rise with $H_{-1}=0$ to the variational operator
\begin{equation}
\CE =\CD_1\vert_{v=u_{xxx}} = 2u_{xxx} D_x + u_{xxxx}.
\label{TwoHD}
\end{equation}
Using the fact that operator $\CE$ in equation \ref{TwoHD} is a symplectic operator, the compatibility condition \ref{CCCH3} gives
\begin{equation}
\begin{aligned}
\CE \left( u_t -  \Eop( H_0) \vert_{v=u_{xxx}}\right) &= 2u_{xxx}u_{tx} + u_{xxxx}u_t - \CE\left(\Eop(H_0) \right)\vert_{v=u_{xxx}} \\
&=\Eop\left( \frac{1}{2} u_{xx}^2 u_t \right).
\end{aligned}
\label{mapleHDold5}
\end{equation}
Equation \ref{mapleHDold5} shows directly that $\CE$ in \ref{TwoHD} is a variational operator for equation \ref{HD3PF}.

It is worth noting that $\CE(K) = 0$ in this example and that 
$[\omega]  = d_\sV [\eta]$ where $[\eta ] \in H^{1,1}(\CR)$. The representative    
$$
\omega = dx \wedge \theta^0 \wedge\epsilon - dt \wedge \bbeta(\epsilon) + d_\sH ( \theta^0 \wedge \theta^1 \cdot u_{xx} ) 
+\frac{1}{3} d_\sH \circ d_\sV(  uu_{xx}\theta^1-(u_x u_{xx}+u u_{xxx})\theta^0 )
$$
with  $ \epsilon =-\frac{1}{2}\CE(\theta^0)= - u_{xxx}\theta^1 -\frac{1}{2}u_{xxxx} \theta^0$ satisfies $\omega = d_\sV \eta$ where (see equation \ref{CF2})
$$
\eta = dx \wedge \theta^0 \cdot\left(-\frac{2}{3}u_x u_{xxx}-\frac{1}{3}uu_{xxxx}\right)- dt \wedge \bbeta(-\frac{2}{3}u_x u_{xxx}-\frac{1}{3}uu_{xxxx}).
$$
Since $d_\sH \eta = 0$, $[\eta] \in H^{1,1}(\CR)$.  This also produces an example where 
$$
Q= -\frac{2}{3}u_x u_{xxx}-\frac{1}{3}uu_{xxxx}
$$
satisfies ${\bf L}_\Delta^*(Q)  = 0$, as well as equation \ref{FQH11}. By Theorem \ref{CLisdV}, Corollary \ref{cordvcl} or Corollary \ref{A5}, $Q$ is not the characteristic of a classical conservation law 
\end{Example}

\begin{Example}\label{ECKDV} The cylindrical KdV equation is (see \cite{wang:2002a}) 
\begin{equation}
v_t = v_{xxx} + v v_x -\frac{v}{2t}
\label{CKDV}
\end{equation}
while it's potential form is
\begin{equation}
u_t = u_{xxx}+\frac{1}{2}u_x^2-\frac{u}{2t}.
\label{PCKdV}
\end{equation}
The form $\kappa$ in equation  \ref{kappadef} in Theorem \ref{FOT} is $\kappa= t^{-1} dt=d_\sH (\log t) $ and so equation \ref{PCKdV} admits $ \CE_1= t D_x$  as a variational operator.   In equation \ref{thm_m} we have $Q_1 = -\frac{1}{2} t u_x$ leading to
$$
\CE_1\left( u_t - u_{xxx}-\frac{1}{2}u_x^2+\frac{u}{2t}\right)=
\Eop\left( -\frac{1}{2}t u_x \left( u_t - u_{xxx}-\frac{1}{2}u_x^2+\frac{u}{2t}\right)-\frac{1}{12}t u_x^3 \right).
$$
By solving the equation $\theta^0 \wedge \CL_\Delta^*(\epsilon)=0$ from  \ref{Urep}
for third order forms $\epsilon$ we find that equation \ref{PCKdV} admits a third order symplectic or variational operator,
$$
\CE_0= t^2D_x^3+\frac{1}{3}(2t^2 u_x+t x)D_x+\frac{1}{6}(2t^2u_{xx}+t).
$$
For $ \CE_0$  we have 
$$
Q_0=-\frac{1}{6}\left(t^2 u_x^2+txu_x+3u_{xxx}t^2\right)
$$
in equation \ref{thm_m} leading to
$$
\CE_0\left( u_t - u_{xxx}-\frac{1}{2}u_x^2+\frac{u}{2t}\right)=\Eop\left( Q_0 \left( u_t - u_{xxx}-\frac{1}{2}u_x^2+\frac{u}{2t}\right)-\frac{1}{72}\left(t^2u_x^4+2txu_x^3\right)\right)
$$

If we now compute the reduction of the potential cylindrical KdV by substituting $w = \sqrt{t}\, u_x$ into
the $x$-derivative of  equation  \ref{PCKdV} we get
\begin{equation}
w_t = w_{xxx} +\frac{1}{\sqrt{t}} w w_x = \CD_1( \Eop(H_1))= \CD_0( \Eop(H_0))
\label{CVCKDV}
\end{equation}
where
\begin{equation}
\CD_1= D_x, \ H_1= \frac{1}{2} w_{x}^2 +\frac{1}{6\sqrt{t}} w^3, \quad  
\CD_0= D_x^3 + \frac{2 w}{3\sqrt{t}}D_x + \frac{w_{x}}{3\sqrt{t}}, \quad H_0 = \frac{1}{2} w^2.
\label{HSCKDV}
\end{equation}
Equation \ref{CVCKDV} can of course be obtained from the cylindrical KdV equation \ref{CKDV} by the change of  variables $w=\sqrt{t}\, v$. It is unclear (to the authors) if the  cylindrical KdV in equation \ref{PCKdV} is  a bi-Hamiltonian evolution equation for which $\CD_1$ and $\CD_0$ in equation \ref{HSCKDV} are Hamiltonian operators. Reference  \cite{wang:2002a} states there are no Hamiltonians for the cylindrical KdV.   It is straightforward to work out the symplectic or variational operators for the potential cylindrical KdV  from equation \ref{CVCKDV} following Theorem \ref{biht}.

More generally for any evolution equation of the form
\begin{equation}
u_t= u_{xxx} +a(t) u_x^2.
\label{genpckdv}
\end{equation}
admits $D_x$ as a first variational operator. We find after a long computation that equation \ref{genpckdv} 
admits a third order variational in the case where $a(t) \dot a(t)\neq 0$ only when
\begin{equation}
a(t) = \pm \frac{1}{\sqrt{ c_1 t +c_2} }.
\label{At}
\end{equation}
For the  $+$ sign in equation \ref{At}, the change of variables $t=c_1^{-1}(\hat t -c_2), \ x=c_1^{-\frac{1}{3}}\hat x $,  $u=\frac{1}{2}c_1^{\frac{1}{3}} \hat u $, takes equation \ref{genpckdv} with $a(t)$ in \ref{At} to the potential form of the cylindrical KdV obtained from equation \ref{CVCKDV}.  The same result holds in the other cases with slightly different changes of variable.

\end{Example}

\section{Conclusions}  The determination of a symplectic operator for a scalar evolution equation has been shown to be equivalent to the existence of a variational operator which is determined by a non-vanishing cohomology class in $H^{1,2}(\CR)$.  The arguments used to prove this clearly extend to other types of differential equations including systems.   In particular Theorem \ref{e_o} holds independently of $\Delta$ being a evolution equation and so the variational operators for $\Delta$ always determine an element of the cohomology $H^{n-1,2}(\Delta)$ as in Theorem \ref{e_o}.

There remain many open theoretical questions such as  whether the compatibility condition for symplectic operators  appears in the cohomology.  Another interesting problem is  to determine under what conditions the
symmetry reduction of a variational operator equation a Hamiltonian system (the converse of Lemma \ref{PFT}).  

Many difficult computational questions have also not been resolved. We were unable to compute the dimension of $H^{1,2}(\CR)$ in our examples. Preliminary computations using equation $\theta^0\wedge {\bf L}_\Delta^*(\rho)=0$ from Theorem \ref{Thmcf1} suggests $\dim H^{1,2}(\CR)= 2$ for the Krichever-Novikov equation in Example 1 and others. However we were not able to give a full proof of this fact.   We have also not explored in any detail the obvious generalization of Noether's Theorem which arises from the existence of a variational operator or equivalently by utilizing a non-trivial element of $H^{1,2}(\CR)$. This would provide an alternate derivation for identifying symmetries and conservation laws for symplectic Hamiltonian systems (see Theorem 7.15 in \cite{Olver:1993a}).

\begin{appendices}
\section{The Vertical Differential}\label{SdV}

The vertical differential induces a mapping $d_\sV: H^{r,s} (\CR) \to H^{r,s+1}(\CR)$ defined
by $d_\sV [\omega] = [ d_\sV \omega]$.  Let $u_t=K$ be a scalar evolution equation with equation manifold. We now examine when $[\omega ] \in {\rm Image} \, d_\sV$.

\begin{Theorem} \label{CLisdV}  Let $[\zeta] \in H^{1,1}(\CR)$.  There exists $[\kappa] \in H^{1,0}(\CR)$ such that 
$[\zeta ] = d_\sV[\kappa]$ if and only if $\delta_\sV \circ \Pi(\zeta)=0$ where $\Pi:\Omega^{1,1}(\CR) \to \Omega^{1,1}_{\rm t_{sb}}(E)$ is defined as in equation \ref{PSE} and $\zeta$ is any representative of $[\zeta]$. 
\end{Theorem}

This answers the question of when $[\zeta]$ is the image of a classical conservation law $[\kappa]$.  To relate Theorem \ref{CLisdV} to the theory of characteristics for a conservation law, suppose $[\zeta]\in H^{1,1}(\CR)$ with (unique) canonical representative given in Theorem \ref{Thmcf2} by
$$
\zeta=dx \wedge \theta^0 \cdot  Q- dt \wedge \bbeta(Q)
$$
where the function $Q$ satisfies $\CL_\Delta^*(Q) = 0$. Theorem \ref{CLisdV} states that the function $Q$ is the characteristic of a classical conservation law for $\Delta$ if and only if $Q=E(L)$.  The test for this condition is the Helmholtz condition $\delta_\sV^\sE ( dx \wedge \theta^0_E \cdot Q) =0$.

\begin{proof}  Suppose $[\zeta]\in  H^{1,1}(\CR) $ where $\zeta=dx \wedge(a_i \theta^i)- dt \wedge \beta$ is a representative, then
$$
I_\sEL \circ d_\sV^\sE \circ \Pi (dx \wedge (a_i\theta^i)-dt \wedge \beta ) =  I_\sEL \circ  d_\sV^\sE \left( dx\wedge (a_i \theta^i_\sE)\right) = 0.
$$
This implies by equation \ref{IEP}, 
\begin{equation}
dx \wedge (a_i \theta_\sE^i)= d_\sV^\sE( g dx) + d_\sH^\sE ( m_i \theta^i_\sE ).
\label{talpha}
\end{equation}

Let $\mu = m_i \theta^i \in \Omega^{1,0}(\CR)$ and
\begin{equation}
\hat \zeta=\zeta - d_\sH \mu.
\label{heta}
\end{equation}
so that $[\hat \zeta]=[\zeta]$.  Now by equation \ref{PXC}, the definition of $\Pi$ in equation \ref{bphi}, and equation \ref{talpha},
$$
\Pi( \hat \zeta) = \Pi(\zeta) - d_\sH^\sE \circ \Pi( \mu) = dx \wedge (a_i \theta_\sE^i)- d_\sH^\sE( m_i \theta^i_\sE)= d_\sV^\sE( g dx ) = g_i \theta^i_\sE.
$$
Therefore there exists $ \hat \beta \in \Omega^{0,1}(\CR)$ such that,
\begin{equation}
\hat \zeta  = g_i \theta^i \wedge dx - dt \wedge \hat \beta= d_\sV (g dx)- dt \wedge \hat \beta.
\label{hzeta}
\end{equation}

Now $ d_\sH d_\sV \hat \zeta = - d_\sV d_\sH \zeta = 0$ and therefore from equation \ref{hzeta},
\begin{equation}
d_\sH d_\sV \hat \zeta = dt \wedge dx \wedge X( d_\sV \beta) =0.
\label{dhdvz}
\end{equation}
However $d_\sV \beta \in \Omega^{0,2}(\CR)$ and the only way the contact two form $d_\sV \beta$ satisfies equation \ref{dhdvz} is if $d_\sV \beta=0$. This implies from equation \ref{hzeta} that $d_\sV \hat \eta= 0$. 

Using the vertical exactness of $\Omega^{1,1}(\CR)$ we conclude there exists $\kappa \in \Omega^{1,0}(\CR)$ such that $\hat \zeta = d_\sV \kappa$. Now 
$$
d_\sV d_\sH \kappa = -d_\sH d_\sV \kappa= -d_\sH \hat \zeta = 0.
$$
Again by vertical exactness of the (augmented) variational bicomplex for $d_\sV : \Omega^{2,0}(\CR) \to \Omega^{2,1}(\CR) $ applied to $d_\sH \kappa$ we have,
$$
d_\sH \kappa =  a(t,x) dt \wedge dx.
$$
Since $\reals^2$ is simply connected we may write 
\begin{equation}
d_\sH \kappa= a(t,x)dt\wedge dx=d( g(t,x) dx +h(t,x) dt).
\label{dhka}
\end{equation}
Finally let 
$$
\hat \kappa  = \kappa - g(t,x) dx -h(t,x) dt,
$$
so that $d_\sV \hat \kappa = d_\sV \kappa = \hat \zeta$, and equation \ref{dhka} gives
$$
d_\sH \hat \kappa=d_\sH \kappa -d(  g(t,x) dx +h(t,x) dt) = 0.
$$
Therefore $ [\zeta]= d_\sV [\hat \kappa]$ and $[\hat \kappa]\in H^{1,0}(\CR)$. 
\end{proof}

\begin{Corollary} \label{cordvcl}  Let $[\zeta]\in H^{1,1}(\CR)$ with
canonical representative given by
$$
\zeta= dx \wedge \theta^0 \cdot Q- dt \wedge \beta
$$
where $\CL_\Delta^*(Q) = 0$ (see Theorem \ref{Thmcf1}). Then $[\zeta]= d_\sV [\kappa]$ where $[\kappa] \in H^{1,0}(\CR)$ if and only if the function $Q$ is in the image of the Euler operator. That is if and only if there exists $A(t,x,u, u_x,u_{xx},\ldots) \in C^\infty( E)$ such that $Q = \Eop(A)$.
\end{Corollary}

\begin{Corollary} If $u_t=K$ is even order, then every solution $Q$ to $\CL_\Delta^*(Q)=0$ is
the characteristic of a conservation law. 
\end{Corollary}

As is well known, the characteristics of a conservation law are solutions to $\CL_\Delta^*(Q)=0$ but the converse is not necessarily true and Corollary \ref{cordvcl} identifies those which are. See Example \ref{HD2} for a solution to $\CL_\Delta^*(Q)=0$ which is not a characteristic of a conservation law.

We now examine the case of $H^{1,2}(\CR)$. 

\begin{Theorem}\label{kerPi} Let $[\omega] \in H^{1,2}(\CR)$.  Then $[\omega]=d_\sV [\eta]$ where $[\eta] \in H^{1,1}(\CR)$ if and only if $[\omega] \in \ker \Lambda$ where $\Lambda: H^{1,2}(\CR) \to H^{2,0}(\CR)$ is defined in equation \ref{defLambda}.
\end{Theorem}

\begin{proof}  Let $[\omega] \in H^{1,2}(\CR)$ with representative $\omega$ satisfying $ \omega = d_\sV \eta$ and $\lambda$ be as in Lemma \ref{exlam}. That is 
$$
[\omega] = [d_\sV \eta] \ \quad d_\sH \eta = d_\sV \lambda.
$$
Suppose now that $\lambda = d_\sH\kappa$ so that $[\omega] \in \ker \Lambda$. Let $\hat \eta= \eta + d_\sV \kappa$. 
Then
$$
[d_\sV \hat \eta]= [d_\sV \eta ], \quad d_\sH \hat \eta= d_\sH \eta + d_\sH d_\sV\kappa= d_\sV \lambda - d_\sV d_\sH \kappa = 0.
$$
Therefore $[\omega] = d_\sV [\hat \eta] $ where $[\hat \eta]\in H^{1,1}(\CR)$. This proves sufficiency of the condition.

Suppose now that $[\omega]= d_\sV [\eta]$ where $[\eta] \in H^{1,1}(\CR)$. Let $\omega$ be the 
representative such that $ \omega = d_\sV \eta$. By hypothesis $d_\sH \eta=0$ and so
for $\lambda $ in Lemma \ref{exlam} we have
$$
d_\sH \eta = d_\sV \lambda = 0.
$$
The same argument as in the second part of the proof of Corollary \ref{LamHom} implies that there exists $\kappa \in \Omega^{1,0}(\reals^2)$ such that  $\lambda  = d_\sH \kappa$. Therefore $\Lambda([\omega]) =[\lambda]= [ d_\sH \kappa] = 0$.
\end{proof}
Theorem \ref{kerPi} is demonstrated in Example \ref{HD2}.

As a simple corollary to Theorem \ref{kerPi}  we can identify the elements of $H^{1,1}(\CR)$ which 
are not the image of a conservation law as follows.

\begin{Corollary} \label{A5} The map $d_\sV : H^{1,1}(\CR)/ d_\sV(H^{1,0}(\CR)) \to \ker \Lambda$ is an isomorphism.
Moreover, we can identify $\eta\in H^{1,1}(\CR)/d_\sV(H^{1,0}(\CR))$ with the space of functions $Q\in C^\infty(\CR)$ such that 
\begin{equation}
({\bf F}_{\tQ}^*- {\bf F}_{\tQ} )( u_t-K) = \Eop\left( \tQ( u_t-K) \right) \qquad    \delta_\sV^\sE ( dx \wedge \theta_\sE \cdot \tQ ) \neq 0 .
\label{FQH11}
\end{equation}
\end{Corollary}

\end{appendices}


\begin{bibdiv}
\begin{biblist}


\bib{vino:1986a}{article}{
author={Astashov, A.M.},
author={Vinogradov, A.M.},
title={On the structure of Hamiltonian operators in field theory},
journal={J. Geom. Pays.},
volume={3},
number={2},
year={1986},
pages={263--287}
} 







\bib{anderson:1992a}{article}{
  author={Anderson, I. M.},
  title={Introduction to the Variational Bicomplex},
  journal={Mathematical Aspects of Classical Field Theory},
  year={1992},
  volume={132},
  pages={51--73}
}

\bib{anderson:2016a}{book}{
  author={Anderson, I. M.},
  title={The Variational Bicomplex},
  }

\bib{anderson-duchamp:1984a}{article}{
 author={Anderson, I. M.}, author={Duchamp, T.E.},
title={Variational principles for second-order quasi-linear scalar equations},
journal={Journal of Differential Equations},
volume={51}, number={1}, year={1984}, pages={1--47}
}



\bib{anderson-kamran:1997a}{article}{
  author={Anderson, I. M.},author={Kamran, N.},
   title={The Variational Bicomplex for Hyperbolic Second-order Scalar Partial Differential Equations in the Plane},
  journal={Duke Math. Jour.},
  year={1997},
  volume={87},number={2},
  pages={265--319}
}


\bib{anderson-thompson:1992a}{article}{
  author={Anderson, I. M.},
  author={Thompson, G.},
  title={The Inverse Problem of the Calculus of Variations for Ordinary Differential Equations},
  journal={Men. Amer. Math. Soc.},
  year={1992},
  volume={98}
 }


\bib{cooke:1991a}{article}{
author={Cooke, D.B.},
title={Compatibility conditions for Hamiltonian pairs},
  journal={Journal of Mathematical Physics},
  volume={32},
year={1991},
pages={3071--3076}
 }

\bib{Douglas:1941a}{article}{
  author={Douglas, J.},
  title={Solution to the inverse problems of the calculus of variations},
  journal={Trans. Amer. Math. Soc.},
  year={1941},
  volume={50},
  pages={71--128}
}

\bib{do-prince:1941a}{article}{
  author={Do, T., Prince, G.},
  title={New progress in the inverse problem in the calculus of variations},
  journal={Differential Geometry and its Applications},
  year={2016},
  volume={45},
  pages={148--179}
}

\bib{dorfman:1993a}{book}{
author={Dorfman, I.},
title={Dirac structures and integrability of nonlinear evolution equations},
publisher = {Wiley and Sons}, 
address = {Hoboken},
year={1993}
 }



\bib{douglas:1941a}{article}{
  author={Douglas, J.},
  title={Solution to the inverse problems of the calculus of variations},
  journal={Trans. Amer. Math. Soc.},
  year={1941},
  volume={50},
  pages={71--128}
}


\bib{fels:1996a}{article}{
  author={Fels, M.E.},
   title={The Inverse Problem of the Calculus of Variations for Scalar Fourth-Order Ordinary Differential Equations},
  journal={Trans. Amer. Math. Soc.},
  year={1996},
  volume={348},
  number={12},
  pages={5007--5029}
}

\bib{fels:2018a}{article}{
  author={Fels, M.E.},
   title={Symplectic Operators for Third Order Ordinary Differential Equation},
  journal={In Preparation}
}



\bib{krupka:1981a}{article}{
  author={Krupka, D.},
   title={On the local structure of the Euler-Lagrange mapping of the calculus of variations},
  journal={Proc. Conf. on Diff. Geom. Appl.},
  year={1981},
  pages={181--188}
}




\bib{Nutku:2002a}{article}{
  author={Nutku, Y.},
  author={Pavlov, M. V.},
  title={Multi-Lagrangians for integrable systems},
  journal={Journal of Mathematical Physics},
  year={2002},
  volume={43},
  pages={1441--1459}
 }

 \bib{Olver:1988a}{article}{
  author={Olver, P.J.},
  title={DarbouxÕ Theorem for Hamiltonian Differential Operators},
  journal={Journal of Differential Equations},
  year={1988},
  volume={77},
  pages={10--33}
 }


\bib{Olver:1993a}{book}{
author={Olver, P. J.},
title={Applications of Lie Groups to Differential Equations; Second Edition},
series = {Graduate Texts in Mathematics},
volume = {107},
publisher = {Springer-Verlag}, 
address = {New York},
year={2000}
}
 
\bib{pavlov:2017a}{article}{
author={Pavlov, M.V.},
author={Vitolo, R.F.},
title={Remarks on the Lagrangian representation of bi-Hamiltonian equations
},
journal={J. Geom. Phys.},
volume={113},
year={2017},
pages={239--249}
} 

 
 \bib{saunders:2010a}{article}{
  author={Saunders, D.J.},
   title={Thirty years of the inverse problem in the calculus of variations},
  journal={Reports on Mathematical Physics},
  year={2010},
  volume={66},
  number={1},
  pages={43--53}
}
 

 \bib{tsujishita:1982a}{article}{
  author={Tsujishita, T.},
   title={On variational bicomplexes associated to differential equations},
  journal={Osaka J. Math.},
  year={1982},
  volume={14},
  number={1},
  pages={311-363}
}
 
 

\bib{vino:1984a}{article}{
author={Vinogradov, A.M.},
title={
The $C$-spectral sequence, Lagrangian formalism, and conservation laws. I, The linear theory},
journal={J. Math. Anal. Appl.},
volume={100},
number={1},
year={1984},
pages={1--40}
} 

\bib{vino:1984b}{article}{
author={Vinogradov, A.M.},
title={
The $C$-spectral sequence, Lagrangian formalism, and conservation laws. II, The non-linear theory},
journal={J. Math. Anal. Appl.},
volume={100},
number={1},
year={1984},
pages={41--129}
} 



\bib{wang:2002a}{article}{
author={Wang, J.P.},
title={A List of 1 + 1 Dimensional Integrable Equations
and Their Properties},
journal={Journal of Nonlinear Mathematical Physics},
volume={9},
number={1},
year={2002},
pages={213--233}
}




\end{biblist}
\end{bibdiv}
\end{document}